\documentclass[11pt]{amsart}
\usepackage{txfonts}
\usepackage{amsfonts}
\usepackage{amssymb}
\usepackage{amsmath}
\usepackage{mathrsfs}
\usepackage{mathtools}
\usepackage{breqn}
\usepackage[pdftex]{hyperref}

\title[The Local Existence and Blowup Criterion]{The Local Existence and Blowup Criterion for Strong Solutions to the Kinetic Cucker--Smale Model Coupled with the Compressible Navier--Stokes Equation}

\author[C. Jin]{Chunyin Jin}
\address[Chunyin Jin]{\newline Beijing International Center for Mathematical Research,\newline Peking University, Beijing 100871, P. R. China.}
\email{jinchunyin@pku.edu.cn}

\newtheorem{theorem}{Theorem}[section]
\newtheorem{definition}{Definition}[section]
\newtheorem{lemma}{Lemma}[section]
\newtheorem{proposition}{Proposition}[section]

\newcommand{\bbr}{\mathbb R}
\newcommand{\bbn}{\mathbb N}

\newcommand{\bbt}{\mathbb T}

\newcommand{\e}{\varepsilon}

\newcommand{\bx}{\mbox{\boldmath $x$}}
\newcommand{\by}{\mbox{\boldmath $y$}}
\newcommand{\bu}{\mbox{\boldmath $u$}}
\newcommand{\bv}{\mbox{\boldmath $v$}}
\newcommand{\bw}{\mbox{\boldmath $w$}}

\newcommand{\bz}{\mbox{\boldmath $z$}}

\begin{document}

\date{\today}



\begin{abstract}
In this paper, we establish the existence and uniqueness of local strong solutions to the kinetic Cucker--Smale model coupled with the isentropic compressible Navier--Stokes equation in the whole space. Moreover, the blowup mechanism for strong solutions to the coupled system is also investigated.
\end{abstract}

\maketitle \centerline{\date}
%
%
\section{Introduction}\label{sec-intro}
\setcounter{equation}{0}
In the present paper, we are concerned with the local existence and blowup criterion for strong solutions to the following kinetic Cucker--Smale model coupled with the isentropic compressible Navier--Stokes equation in the whole space $\bbr^3$.
\begin{equation} \label{eq-cs-ns}
     \begin{dcases}
         f_t + \bv \cdot \nabla_{\bx} f+ \nabla_{\bv} \cdot (L[f]f+(\bu-\bv)f)=0,\\
         \rho_t+\nabla \cdot (\rho \bu)=0,\\
         (\rho \bu)_t+\nabla \cdot (\rho \bu \otimes \bu)+\nabla P=\mu \Delta \bu +(\mu+\lambda)\nabla \nabla \cdot \bu+\int_{\bbr^3}f(\bv-\bu)d\bv,
     \end{dcases}
\end{equation}
subject to the initial data
\begin{equation} \label{eq-sys-inidata}
  f|_{t=0}=f_0, \quad \rho|_{t=0}=\rho_0, \quad \bu|_{t=0}=\bu_0.
\end{equation}
Here $f(t,\bx, \bv)$ is the particle distribution function in phase space $(\bx, \bv)$ at the time $t$, $\bx=(x_1,x_2,x_3)\in \bbr^3$, $\bv=(v_1,v_2,v_3) \in \bbr^3$. $\rho$ and $\bu=(u_1, u_2 ,u_3)$ denote the fluid density and velocity, respectively. The constant viscosity $\mu$ and $\lambda$ satisfy the physical restriction
\[
  \mu>0, \qquad 2\mu +3\lambda \ge 0.
\]
The pressure $P$ and $L[f]$ are respectively given by
\[
  P=\rho^{\gamma}, \quad \gamma>1;
\]
\[
 L[f](t, \bx, \bv)=\int_{\bbr^{6}}\varphi(|\bx-\by|)f(t, \by, \bv^*)(\bv^*-\bv)d \by d \bv^*,
\]
where $\varphi(\cdot)$ is a positive non-increasing function representing the interaction kernel. For convenience, we suppose $\varphi \in C^{\infty}$. If not, we mollify it by convolution. In fact, we only need $\varphi \in C^{1}$.  Without loss of generality, we postulate that
\[
 \max\{|\varphi|, |\varphi'|\} \le 1
\]
in the sequel.

The first equation in \ref{eq-cs-ns} is the kinetic Cucker--Smale model derived from the particle model by taking the mean-field limit; see \cite{Carrillo2010}\cite{Ha2009}. The well-posedness of measure-valued solution was also known in \cite{Carrillo2010}\cite{Ha2009}. As for weak and strong solutions in regular function space, Jin \cite{jin2018well} recently established the well-posedness by developing a unified framework. As a fact, an ensemble of particles is usually immersed in ambient fluid, such as gas and water. In order to render the model more realistic, it is natural to incorporate the influence of fluids. Such coupled models have been investigated in the space-periodic domain \cite{Bae2012}\cite{bae2014asymptotic}\cite{bae2014global}, however under strong regularity conditions on the initial data. Besides, the restriction that the interaction kernel $\varphi$ should have a positive lower bound in the torus $\bbt^3$, was crucially used in the analysis of time-asymptotic flocking behaviors for the coupled system. Here in this paper, we will contribute a study on the whole space situation, under a relaxed regularity condition on the initial data. If the Brownian effect is taken into account in the modeling, then the resulting model becomes of the Vlasov--Fokker--Planck type. This type of model allows for equilibrium states, i.e., steady solutions. Duan \cite{duan2010kinetic} studied the stability around a equilibrium under small initial perturbations. The same type results for coupled models with fluids were also obtained in \cite{carrillo2011global}\cite{goudon2010navier}\cite{li2017strong}, by using the micro-macro decomposition. For the hydrodynamic Cucker--Smale model and related coupled models with fluids, we refer to \cite{ha2014global}\cite{Ha2014}\cite{ha2015emergent}\cite{jin2015well}\cite{Jin}. The interested reader can also consult the review papers \cite{carrillo2010particle}\cite{choi2017emergent} for the state of the art in this territory.

The rest two equations in \eqref{eq-cs-ns} are the isentropic compressible Navier--Stokes equation with the coupling term. For the multi-dimensional compressible Navier--Stokes equation, the local existence and uniqueness were obtained in \cite{nash1962probleme}\cite{tani1977first}, when the initial density was away from vacuum, i.e., the initial density had a positive lower bound. The global-in-time classical solutions was first constructed in \cite{matsumura1980initial} around a non-vacuum equilibrium, under small initial perturbations in $H^3$. As regard the global existence for large data, the breakthrough was due to Lions \cite{lions1996mathematicalv2}, where the finite energy weak solutions were obtained when $P=\rho^{\gamma}$ $(\gamma>\frac95)$, by means of the weak convergence method. Using the framework in \cite{lions1996mathematicalv2}, Feireisl \cite{feireisl2004dynamics} further relaxed the restriction on $\gamma$ to $\gamma>\frac32$. However, the uniqueness and regularity of weak solutions are still unknown until now. Xin \cite{xin1998blowup} first investigated the blowup mechanism for the classical solutions to the compressible Navier--Stokes equation with compactly supported initial density. Later, Huang, Li and Xin et al. \cite{huangli2013serrin}\cite{huang2013serrin}\cite{huang2011blowup}\cite{huang2011serrin} established a series of blowup criterions for the isentropic compressible Navier--Stokes equation, full compressible Navier--Stokes equation and MHD models, by ingeniously using Beal--Kato--Majda's logarithmic inequality. Based on their previous analyses on the blowup mechanism, together with the recent study \cite{Cho2006a} on local-in-time classical solutions to the compressible Navier--Stokes equation  with nonnegative initial densities, Huang--Li--Xin \cite{huang2012global} and Huang--Li \cite{huang2018global} successfully obtained the global-in-time classical solutions for the  isentropic compressible Navier--Stokes equation and full compressible Navier--Stokes equation, when the initial data satisfied some regularity and compatibility conditions, and the initial energies were suitably small. The key to both proofs was to derive the uniform bound on the density.

Combining our analysis on the kinetic Cucker--Smale model with the recent development in the compressible Navier--Stokes equation, it is shown that the difficulty in this paper is to tackle the coupling term. In order to overcome the hard estimates arising from the coupling term, we introduce a weighted Sobolev norm for $f(t,\cdot,\cdot)$. It turns out that all the $L^p$-type norms for $f(t,\bx,\bv)$ with respect to $\bx$ and $\bv$ can be controlled by the introduced weighted Sobolev norm. The weighted Sobolev space is defined as follows:
\begin{multline*}
 H_{\omega}^1(\bbr^3 \times \bbr^3):=\bigg\{h(\bx,\bv):\ h \in L_{\omega}^2(\bbr^3 \times \bbr^3), \\ \nabla_{\bx}h\in L_{\omega}^2(\bbr^3 \times \bbr^3),\ \nabla_{\bv}h \in L_{\omega}^2(\bbr^3 \times \bbr^3)\bigg\},
\end{multline*}
\[
 |h|_{H_{\omega}^1}^2:=|h|_{L_{\omega}^2}^2+|\nabla_{\bx}h|_{L_{\omega}^2}^2+|\nabla_{\bv}h|_{L_{\omega}^2}^2,
\]
where
\[
 |h|_{L_{\omega}^2}:=\left(\int_{\bbr^6}h^2(\bx,\bv)\omega(\bx,\bv)d\bx d\bv\right)^{\frac12},
\]
and
\[
 \omega(\bx,\bv):=(1+\bv^2)^{2+\beta}(1+\bx^2+\bv^2)^{\alpha},\quad \alpha>3,\  \beta>\frac12.
\]
In this paper, we adopt the following simplified notations for homogeneous Sobolev Spaces.
\[
 D^1(\bbr^3):=\left\{u\in L^6(\bbr^3):\ \nabla u \in L^2(\bbr^3)\right\},
\]
\[
 D^2(\bbr^3):=\left\{u\in L_{loc}^1(\bbr^3):\ \nabla^2u \in L^2(\bbr^3)\right\},
\]
\[
 D^{2,p}(\bbr^3):=\left\{u\in L_{loc}^1(\bbr^3):\ \nabla^2u \in L^p(\bbr^3)\right\}, \quad 1\le p\le \infty.
\]
Next we give the definition of strong solutions to \eqref{eq-cs-ns}.
\begin{definition} \label{def-stro}
Let $3<q\le 6$. $(f(t,\bx,\bv), \rho(t,\bx),\bu(t,\bx))$ is said to be a strong solution to \eqref{eq-cs-ns} in $[0,T]\times \bbr^3\times \bbr^3$, if
\[
\begin{aligned}
  &f(t,\bx,\bv)\in C([0,T];H_{\omega}^1(\bbr^3 \times \bbr^3)),\\
  &\rho(t,\bx)\in C([0,T];H^1\cap W^{1,q}(\bbr^3)),\\
  &\bu(t,\bx)\in C([0,T];D^1\cap D^{2}(\bbr^3))\cap L^2(0,T;D^{2,q}(\bbr^3)),\\
  &\rho_t(t,\bx)\in C([0,T];L^2\cap L^{q}(\bbr^3)),\quad \bu_t(t,\bx)\in L^2(0,T;D^{1}(\bbr^3)),
\end{aligned}
\]
and \eqref{eq-cs-ns} holds in the sense of distributions.
\end{definition}

Denote by $B(R_0)$ the ball centered at the origin with a radius $R_0$. Then the theorems in this paper can be stated as follows.
\begin{theorem}[Local existence] \label{thm-loc-exist}
Let $3<q\le 6$, $R_0>0$. Assume the initial data $f_0(\bx,\bv)\ge 0$, $\rho_0(\bx)\ge 0$ and $\bu_0(\bx)$ satisfy
\[
 f_0(\bx,\bv) \in H_{\omega}^1(\bbr^3 \times \bbr^3), \quad \rho_0(\bx) \in H^1\cap W^{1,q}(\bbr^3), \quad \bu_0(\bx) \in D^1\cap D^{2}(\bbr^3)
\]
with the compatibility condition
\[
 \rho_0^{\frac12}g+\nabla P(\rho_0)=\mu \Delta \bu_0+(\mu+\lambda)\nabla \nabla \cdot \bu_0+\int_{\bbr^3}f_0(\bv-\bu_0)d \bv, \quad g \in L^2(\bbr^3),
\]
and the initial support of $f_0(\bx,\bv)$ with respect to $\bv$
\[
 \text{supp}_{\bv}f_0(\bx,\cdot)\subseteq B(R_0) \quad \text{for all $\bx \in \bbr^3$}.
\]
Then the Cauchy problem \eqref{eq-cs-ns}-\eqref{eq-sys-inidata} admits a unique local strong solution in the sense of Definition \ref{def-stro}.
\end{theorem}

\begin{theorem}[Blowup criterion]\label{thm-blowup}
Under the conditions in Theorem \ref{thm-loc-exist}, assume $f(t,\bx,\bv)$ in $[0,T^*)\times\bbr^3 \times \bbr^3$ and $(\rho(t,\bx),\bu(t,\bx))$ in $[0,T^*)\times\bbr^3$ are the strong solutions to \eqref{eq-cs-ns}-\eqref{eq-sys-inidata} in the sense of Definition \ref{def-stro}. If the life span $T^*<\infty$, then
\[
 \lim_{T \to T^*}\left(|\rho(t,\bx)|_{L^{\infty}(0,T;L^{\infty})}+\int_0^T(|\bu(t)|_{L^{\infty}}+|\nabla \bu(t)|_{L^{\infty}}^2) dt \right)=\infty.
\]
\end{theorem}

The rest of the paper is organized as follows. In Section 2, we present a detailed analysis on the kinetic Cucker--Smale model in the weighted Sobolev space. In Section 3, we construct the local strong solution to the coupled system. Section 4 is devoted to the investigation on blowup mechanisms for the strong solutions to the coupled model. In the last section, we summarize the paper and make a comment on it.
\vskip 0.3cm
\noindent\textbf{Notation}. Throughout the paper, $C$ represents a general positive constant that may depend on $\mu$, $\lambda$, $\gamma$, $\varphi$, $\varphi'$ and the initial data. We write $C(*)$ to emphasize that $C$ additionally depends on *. Both $C$ and $C(*)$ may differ from line to line. $\nabla$ and $\partial_i$ are abbreviated for $\nabla_{\bx}$ and $\partial_{x_i}$, $1\le i \le 3$, respectively. The Einstein summation convention is also used in the paper.

%
%
\section{Well-posedness of the Kinetic Cucker-Smale Model}\label{sec-kine-cs}
\setcounter{equation}{0}
In this section, we study the well-posedness of the kinetic Cucker--Smale  model $\eqref{eq-cs-ns}_1$ in the weighted Sobolev space. Consider
\begin{equation} \label{eq-kine-cs}
     \begin{dcases}
         f_t + \bv \cdot \nabla_{\bx} f+ \nabla_{\bv} \cdot (L[f]f+(\bu-\bv)f)=0,\\
         f|_{t=0}=f_0(\bx,\bv),
     \end{dcases}
\end{equation}
for given $\bu(t,\bx)\in C([0,T];D^1\cap D^{2}(\bbr^3))\cap L^2(0,T;D^{2,q}(\bbr^3))$, $3<q\le 6$.
Define the bound of $\bv$-support of $f(t,\bx, \bv)$ at the time $t$ as
\[
  R(t):=\sup \Big\{|\bv|: \ (\bx,\bv) \in \text{supp} f(t,\cdot,\cdot)\Big\}.
\]
The result in this section is summarized as follows.
\begin{proposition}\label{prop-kine-cs-wp}
Let $R_0>0$, $T>0$. Assume $0\le f_0(\bx,\bv) \in H_{\omega}^1(\bbr^3 \times \bbr^3)$, and $\text{supp}_{\bv}f_0(\bx,\cdot)\subseteq B(R_0)$ for all $\bx \in \bbr^3$. Given $\bu(t,\bx)\in C([0,T];D^1\cap D^{2}(\bbr^3))\cap L^2(0,T;D^{2,q}(\bbr^3))$, $3<q\le 6$, there exists a unique non-negative strong solution $f(t,\bx,\bv)\in C([0,T];H_{\omega}^1(\bbr^3 \times \bbr^3))$ to \ref{eq-kine-cs}. Moreover,
\[
 \begin{aligned}
   &(1)\ R(t)\le R_0+\bigg[|f_0|_{L^1}\Big(R_0+C\sup_{0\le t \le T}|\bu(t)|_{D^1 \cap D^2}\Big)+C\sup_{0\le t \le T}|\bu(t)|_{D^1 \cap D^2}\bigg]t,\quad 0\le t \le T ;\\
   &(2)\ |f(t)|_{H_{\omega}^1}\le |f_0|_{H_{\omega}^1}\exp \left( C\int_0^t(1+R(\tau)+|\bu(\tau)|_{L^{\infty}}+|\nabla\bu(\tau)|_{L^{\infty}})d \tau \right), \quad 0\le t \le T;\\
   &(3)\ \sup_{0\le t \le T}|f(t)-\tilde{f}(t)|_{L^2_{\omega}}\le \left(|f_0-\tilde{f}_0|_{L^2_{\omega}}+\sup_{0\le t \le T}|\bu (t)-\tilde{\bu}(t)|_{D^1 \cap D^2}\int_0^T C|\nabla_{\bv}\tilde{f}(t)|_{L^2_{\omega}}dt\right)\\
    &\qquad \qquad \qquad \qquad \qquad  \times \exp \left( \int_0^T C(1+|\bu(t)|_{D^1 \cap D^2}+R(t)|\tilde{f}(t)|_{H_{\omega}^1})dt \right),
 \end{aligned}
\]
where $\tilde{f}(t, \bx, \bv)$ is the strong solution to \eqref{eq-kine-cs}, with $\bu$ and $f_0$ replaced by $\tilde{\bu}$ and $\tilde{f}_0$, respectively.
\end{proposition}

In order to prove Proposition \ref{prop-kine-cs-wp}, we need the following a priori estimates.
\subsection{A priori estimates}
\begin{lemma}\label{lm-kine-cs-apriori}
Let $R_0>0$, $T>0$. Assume $0\le f_0(\bx,\bv) \in H_{\omega}^1(\bbr^3 \times \bbr^3)$, and $\text{supp}_{\bv}f_0(\bx,\cdot)\subseteq B(R_0)$ for all $\bx \in \bbr^3$. Given $\bu(t,\bx)\in C([0,T];D^1\cap D^{2}(\bbr^3))\cap L^2(0,T;D^{2,q}(\bbr^3))$, $3<q\le 6$, if $f(t,\bx,\bv)$ is a smooth solution to \ref{eq-kine-cs}, then
\[
 \begin{aligned}
   &(1)\ R(t)\le R_0+\bigg[|f_0|_{L^1}\Big(R_0+C\sup_{0\le t \le T}|\bu(t)|_{D^1 \cap D^2}\Big)+C\sup_{0\le t \le T}|\bu(t)|_{D^1 \cap D^2}\bigg]t,\quad 0\le t \le T ;\\
   &(2)\ |f(t)|_{H_{\omega}^1}\le |f_0|_{H_{\omega}^1}\exp \left( C\int_0^t(1+R(\tau)+|\bu(\tau)|_{L^{\infty}}+|\nabla\bu(\tau)|_{L^{\infty}})d \tau \right), \quad 0\le t \le T;\\
   &(3)\ \sup_{0\le t \le T}|f(t)-\tilde{f}(t)|_{L^2_{\omega}}\le \left(|f_0-\tilde{f}_0|_{L^2_{\omega}}+\sup_{0\le t \le T}|\bu (t)-\tilde{\bu}(t)|_{D^1 \cap D^2}\int_0^T C|\nabla_{\bv}\tilde{f}(t)|_{L^2_{\omega}}dt\right)\\
    &\qquad \qquad \qquad \qquad \qquad  \times \exp \left( \int_0^T C(1+|\bu(t)|_{D^1 \cap D^2}+R(t)|\tilde{f}(t)|_{H_{\omega}^1})dt \right),
 \end{aligned}
\]
where $\tilde{f}(t, \bx, \bv)$ is the smooth solution to \eqref{eq-kine-cs}, with $\bu$ and $f_0$ replaced by $\tilde{\bu}$ and $\tilde{f}_0$, respectively.
\end{lemma}
\begin{proof}
(1)~From $f_0 \in H^1_{\omega}$, we deduce that
\begin{equation} \label{eq-ini-f-lonenm}
 \begin{aligned}
 |f_0|_{L^1}&=\int_{\bbr^6}f_0(\bx,\bv)\omega^{\frac12}(\bx,\bv)\omega^{-\frac12}(\bx,\bv)d\bx d\bv\\
   &\le |\omega^{-\frac12}|_{L^2}|f_0|_{L^2_{\omega}}\le C|f_0|_{L^2_{\omega}}.
 \end{aligned}
\end{equation}
Integrating $\eqref{eq-kine-cs}_1$ over $[0,t]\times \bbr^3 \times \bbr^3$ $(0<t\le T)$ gives
\begin{equation} \label{eq-kin-conser-mass}
 |f(t)|_{L^1}=|f_0|_{L^1}.
\end{equation}
Denote by $(X(t;\bx_0,\bv_0),V(t;\bx_0,\bv_0))$ the characteristic issuing from $(\bx_0,\bv_0)$. It satisfies
\begin{equation} \label{eq-charac}
\begin{dcases}
  \frac{d X}{d t}=V, \\
  \frac{d V}{d t}=\int_{\bbr^{2d}} \varphi(|X-\by|)f(t, \by,\bv^*)(\bv^*-V)d \by d \bv^*.
\end{dcases}
\end{equation}
Define
\[
 a(t,\bx):=\int_{\bbr^{2d}} \varphi(|\bx-\by|)f(t, \by,\bv^*) d \by d \bv^*,
\]
\[
 \mathbf{b} (t,\bx):=\int_{\bbr^{2d}} \varphi(|\bx-\by|)f(t, \by,\bv^*) \bv^* d \by d \bv^*.
\]
Solving the equation \eqref{eq-kine-cs} by the method of characteristics gives
\begin{equation} \label{eq-kin-cs-positive}
 f(t,X(t;\bx_0,\bv_0),V(t;\bx_0,\bv_0))=f_0(\bx_0,\bv_0)\exp \left(3 \int_0^t [1+a(\tau,X(\tau))] d \tau \right)\ge 0.
\end{equation}
Multiplying $\eqref{eq-kine-cs}_1$ by $\bv^2$, we obtain
\begin{equation} \label{eq-kin-cs-ener}
  \frac{\partial}{\partial t} (f \bv^2)+ \bv \cdot \nabla_{\bx}(f \bv^2)+\nabla_{\bv} \cdot \Big(L[f]f \bv^2+(\bu-\bv)f \bv^2\Big)=2fL[f]\cdot \bv+2f \bv \cdot (\bu-\bv).
\end{equation}
Using \eqref{eq-kin-cs-positive} and integrating \eqref{eq-kin-cs-ener} over $\bbr^3 \times \bbr^3$ lead to
\[
 \begin{aligned}
   \frac{d}{dt}\int_{\bbr^6}f\bv^2 d\bx d\bv=&-\int_{\bbr^{12}}\varphi(|\bx-\by|)f(t,\bx,\bv)f(t,\by,\bv^*)(\bv^*-\bv)^2d\by d\bv^* d\bx d\bv\\
     &-2\int_{\bbr^6}f\bv^2 d\bx d\bv+2\int_{\bbr^3}\int_{\bbr^3}f \bv d\bv \cdot \bu d\bx\\
     \le&-2\int_{\bbr^6}f\bv^2 d\bx d\bv+2|\bu(t)|_{L^{\infty}}|f|_{L^1}^{\frac12}\Bigg(\int_{\bbr^6}f\bv^2 d\bx d\bv\Bigg)^{\frac12}.
 \end{aligned}
\]
Solving the above Gronwall's inequality yields
\begin{equation}\label{eq-cs-twmonieq}
 \Bigg(\int_{\bbr^6}f(t, \bx, \bv)\bv^2 d\bx d\bv\Bigg)^{\frac12}\le |f_0|_{L^1}^{\frac12}\bigg(R_0+C\sup_{0\le t \le T}|\bu(t)|_{D^1 \cap D^2}\bigg), \quad 0\le t \le T,
\end{equation}
where we have employed \eqref{eq-kin-conser-mass} and the following Sobolev inequality
\begin{equation} \label{eq-sobol-infty}
 |\bu(t)|_{L^{\infty}} \le C|\bu(t)|_{D^1 \cap D^2} \quad \text{ in $\bbr^3$}.
\end{equation}
It follows from the characteristic equation \eqref{eq-charac} that
\begin{equation} \label{eq-exnV}
 \begin{aligned}
 V(t)=& V_0 e^{-\int_0^t [1+a(\tau,X(\tau))] d \tau}\\
 &+e^{-\int_0^t  [1+a(\tau,X(\tau))] d \tau }\int_0^t [\mathbf{b}(\tau,X(\tau))+\bu(\tau,X(\tau))] e^{\int_0^{\tau} [1+a(s,X(s))] ds} d \tau.
 \end{aligned}
\end{equation}
Using Cauchy's inequality, we have by \eqref{eq-kin-conser-mass} and \eqref{eq-cs-twmonieq}
\begin{equation}
 \begin{aligned}
   |\mathbf{b}(t,\bx)|\le & \left(\int_{\bbr^{2d}} f(t, \by,\bv^*) d \by d \bv^* \right)^{\frac12}\left(\int_{\bbr^{2d}} f(t, \by,\bv^*) |\bv^*|^2 d \by d \bv^* \right)^{\frac12}\\
   \le &|f_0|_{L^1}\bigg(R_0+C\sup_{0\le t \le T}|\bu(t)|_{D^1 \cap D^2}\bigg), \quad 0\le t \le T.
 \end{aligned}
\end{equation}
This together with \eqref{eq-sobol-infty} and \eqref{eq-exnV} gives
\[
 R(t)\le R_0+\bigg[|f_0|_{L^1}\Big(R_0+C\sup_{0\le t \le T}|\bu(t)|_{D^1 \cap D^2}\Big)+C\sup_{0\le t \le T}|\bu(t)|_{D^1 \cap D^2}\bigg]t,\quad 0\le t \le T.
\]
(2)~Multiplying $\eqref{eq-kine-cs}_1$ by $2f\omega$, we obtain
\begin{equation}\label{eq-zerowt-dif}
  \begin{gathered}
   \frac{\partial}{\partial t}(f^2\omega)+\bv \cdot \nabla_{\bx}(f^2\omega)+\nabla_{\bv}\cdot \Big(L[f]f^2\omega+(\bu-\bv)f^2\omega \Big)=\bv \cdot \nabla_{\bx}\omega f^2\\+L[f]\cdot \nabla_{\bv}\omega f^2 -\nabla_{\bv}\cdot L[f] f^2\omega +(\bu-\bv)\cdot \nabla_{\bv}\omega f^2 +3 f^2\omega.
  \end{gathered}
\end{equation}
Integrating \eqref{eq-zerowt-dif} over $\bbr^3 \times \bbr^3$ yields
\begin{equation}\label{eq-zerowt-gron}
  \begin{aligned}
    \frac{d}{d t}|f(t)|_{L^2_{\omega}}^2 \le C|f(t)|_{L^2_{\omega}}^2+C|\bu(t)|_{L^{\infty}}|f(t)|_{L^2_{\omega}}^2+C|f(t)(1+\bv^2)^{\frac12}|_{L^1}
    |f(t)|_{L^2_{\omega}}^2.
  \end{aligned}
\end{equation}
Applying $\nabla_{\bx}$ to $\eqref{eq-kine-cs}_1$, we deduce that
\begin{equation}\label{eq-fistx-dif}
  \begin{gathered}
  (\nabla_{\bx}f)_t+\bv \cdot \nabla_{\bx}\nabla_{\bx}f+\nabla_{\bv}\cdot \Big(L[f]\otimes \nabla_{\bx}f+(\bu-\bv)\otimes \nabla_{\bx}f \Big)\\
     +\nabla_{\bx}\nabla_{\bv}\cdot L[f]f+ \nabla_{\bx}L[f]\cdot \nabla_{\bv}f+\nabla \bu \cdot \nabla_{\bv}f=0.
  \end{gathered}
\end{equation}
Multiplying \eqref{eq-fistx-dif} by $2\omega \nabla_{\bx}f$, we have
\begin{equation}\label{eq-fistxwt-dif}
  \begin{aligned}
    &(|\nabla_{\bx}f|^2\omega)_t+\bv \cdot \nabla_{\bx}(|\nabla_{\bx}f|^2\omega)+\nabla_{\bv}\cdot \Big(L[f]|\nabla_{\bx}f|^2\omega +(\bu-\bv)|\nabla_{\bx}f|^2\omega \Big)\\
    =&\bv \cdot \nabla_{\bx}\omega |\nabla_{\bx}f|^2-2\omega \nabla_{\bx}f \cdot \nabla_{\bx}L[f]\cdot \nabla_{\bv}f+L[f] \cdot \nabla_{\bv}\omega |\nabla_{\bx}f|^2\\
    &-\omega |\nabla_{\bx}f|^2\nabla_{\bv}\cdot L[f]-2 \omega f \nabla_{\bx}f \cdot \nabla_{\bx}\nabla_{\bv}\cdot L[f]-2 \omega \nabla_{\bx}f \cdot \nabla \bu \cdot \nabla_{\bv}f\\
    &+(\bu-\bv) \cdot \nabla_{\bv}\omega |\nabla_{\bx}f|^2+3|\nabla_{\bx}f|^2\omega.
  \end{aligned}
\end{equation}
Integrating \eqref{eq-fistxwt-dif} over $\bbr^3 \times \bbr^3$ leads to
\begin{equation}\label{eq-fistxwt-gron}
  \begin{aligned}
    \frac{d}{dt}|\nabla_{\bx}f(t)|_{L^2_{\omega}}^2 \le & C|\nabla_{\bx}f(t)|_{L^2_{\omega}}^2 +CR(t)|f(t)|_{L^1}|\nabla_{\bx}f(t)|_{L^2_{\omega}} |\nabla_{\bv}f(t)|_{L^2_{\omega}}\\
    &+C|f(t)(1+\bv^2)^{\frac12}|_{L^1}|\nabla_{\bx}f(t)|_{L^2_{\omega}}^2 +C|f(t)|_{L^1}|f(t)|_{L^2_{\omega}} |\nabla_{\bx}f(t)|_{L^2_{\omega}}\\
    &+C|\nabla \bu(t)|_{L^{\infty}} |\nabla_{\bx}f(t)|_{L^2_{\omega}} |\nabla_{\bv}f(t)|_{L^2_{\omega}} +C|\bu(t)|_{L^{\infty}} |\nabla_{\bx}f(t)|_{L^2_{\omega}}^2.
  \end{aligned}
\end{equation}
Applying $\nabla_{\bv}$ to $\eqref{eq-kine-cs}_1$, we obtain
\begin{equation}\label{eq-fistv-dif}
  \begin{gathered}
  (\nabla_{\bv}f)_t +\bv \cdot \nabla_{\bx}\nabla_{\bv}f+\nabla_{\bv}\cdot \Big(L[f]\otimes \nabla_{\bv}f+(\bu-\bv)\otimes \nabla_{\bv}f \Big)\\
     =-\nabla_{\bv}L[f]\cdot \nabla_{\bv}f +\nabla f - \nabla_{\bx}f.
  \end{gathered}
\end{equation}
Multiplying \eqref{eq-fistv-dif} by $2\omega \nabla_{\bv}f$, we arrive at
\begin{equation}\label{eq-fistvwt-dif}
  \begin{aligned}
    &(|\nabla_{\bv}f|^2\omega)_t+\bv \cdot \nabla_{\bx}(|\nabla_{\bv}f|^2\omega)+\nabla_{\bv}\cdot \Big(L[f]|\nabla_{\bv}f|^2\omega +(\bu-\bv)|\nabla_{\bv}f|^2\omega \Big)\\
    =&-2\omega \nabla_{\bx}f \cdot \nabla_{\bv}f +\bv \cdot \nabla_{\bx}\omega |\nabla_{\bv}f|^2-2\omega \nabla_{\bv}f \cdot \nabla_{\bv}L[f]\cdot \nabla_{\bv}f +L[f] \cdot \nabla_{\bv}\omega |\nabla_{\bv}f|^2\\
    &-|\nabla_{\bv}f|^2 \omega  \nabla_{\bv} \cdot L[f] +(\bu-\bv) \cdot \nabla_{\bv}\omega |\nabla_{\bv}f|^2 +5|\nabla_{\bv}f|^2\omega.
  \end{aligned}
\end{equation}
Integrating \eqref{eq-fistvwt-dif} over $\bbr^3 \times \bbr^3$ results in
\begin{equation}\label{eq-fistvwt-gron}
  \begin{aligned}
    \frac{d}{dt}|\nabla_{\bv}f(t)|_{L^2_{\omega}}^2 \le & C|\nabla_{\bx}f(t)|_{L^2_{\omega}} |\nabla_{\bv}f(t)|_{L^2_{\omega}} +C|\nabla_{\bv}f(t)|_{L^2_{\omega}}^2 \\
    &+C|f(t)(1+\bv^2)^{\frac12}|_{L^1}|\nabla_{\bv}f(t)|_{L^2_{\omega}}^2 +C|\bu(t)|_{L^{\infty}} |\nabla_{\bv}f(t)|_{L^2_{\omega}}^2.
  \end{aligned}
\end{equation}
Adding \eqref{eq-zerowt-gron}, \eqref{eq-fistxwt-gron} and \eqref{eq-fistvwt-gron} together, we infer by \eqref{eq-kin-conser-mass} that
\begin{equation}\label{eq-sobowt-gron}
  \begin{aligned}
    \frac{d}{dt}|f(t)|_{H^1_{\omega}} \le & C\left(1+ C|f(t)(1+\bv^2)^{\frac12}|_{L^1} +R(t)|f(t)|_{L^1} +|\bu(t)|_{L^{\infty}} +|\nabla \bu(t)|_{L^{\infty}}\right) |f(t)|_{H^1_{\omega}}\\
    \le & C\Big(1 +R(t) +|\bu(t)|_{L^{\infty}} +|\nabla \bu(t)|_{L^{\infty}}\Big) |f(t)|_{H^1_{\omega}}
  \end{aligned}
\end{equation}
Solving the above Gronwall inequality gives
\begin{equation}\label{eq-sobowt-norm}
  |f(t)|_{H_{\omega}^1}\le |f_0|_{H_{\omega}^1}\exp \left( C\int_0^t(1+R(\tau)+|\bu(\tau)|_{L^{\infty}}+|\nabla\bu(\tau)|_{L^{\infty}})d \tau \right), \quad 0\le t \le T.
\end{equation}
(3)~ Define
\[
 \overline{f}:=f-\tilde{f}, \quad \overline{\bu}:=\bu-\tilde{\bu}, \quad \overline{f}_0:=f_0-\tilde{f}_0.
\]
It follows from the equation $\eqref{eq-kine-cs}_1$ that
\begin{equation} \label{eq-kine-dfr-cs}
 \overline{f}_t + \bv \cdot \nabla_{\bx} \overline{f}+ \nabla_{\bv} \cdot \Big(L[f]\overline{f}+(\bu-\bv)\overline{f}\Big) =-\nabla \cdot (L[\overline{f}]\tilde{f}) -\nabla \cdot (\overline{\bu}\tilde{f}).
\end{equation}
Multiplying \eqref{eq-kine-dfr-cs} by $2 \overline{f} \omega$, we deduce that
\begin{equation} \label{eq-dfrwt}
 \begin{aligned}
   &(\overline{f}^2 \omega)_t + \bv \cdot \nabla_{\bx} (\overline{f}^2 \omega)+ \nabla_{\bv} \cdot \Big(L[f]\overline{f}^2 \omega +(\bu-\bv)\overline{f}^2 \omega \Big)\\
   =&\bv \cdot \nabla_{\bx}\omega \overline{f}^2-\overline{f}^2 \omega \nabla_{\bv} \cdot L[f] +\overline{f}^2 L[f] \cdot \nabla_{\bv}\omega +3\overline{f}^2 \omega\\
   &+(\bu-\bv) \cdot \nabla_{\bv} \omega \overline{f}^2-2 \omega \overline{f}\Big[\nabla_{\bv} \cdot L[\overline{f}]\tilde{f} +L[\overline{f}] \cdot \nabla_{\bv}\tilde{f} \Big] -2 \omega \overline{f} \overline{\bu} \cdot \nabla_{\bv} \tilde{f}.
 \end{aligned}
\end{equation}
Integrating \eqref{eq-dfrwt} over $\bbr^3 \times \bbr^3$ gives
\begin{equation} \label{eq-dfrwt-gron}
 \begin{aligned}
   \frac{d}{dt}|\overline{f}(t)|_{L^2_{\omega}}^2
   \le & C |\overline{f}(t)|_{L^2_{\omega}}^2 +C |f(t)(1+\bv^2)^{\frac12}|_{L^1}|\overline{f}(t)|_{L^2_{\omega}}^2\\ &+C|\bu(t)|_{L^{\infty}}|\overline{f}(t)|_{L^2_{\omega}}^2
      +C|\tilde{f}(t)|_{L^2_{\omega}} |\overline{f}(t)|_{L^2_{\omega}}^2\\
    &+CR(t)|\nabla_{\bv} \tilde{f}(t)|_{L^2_{\omega}} |\overline{f}(t)|_{L^2_{\omega}}^2 +C|\overline{\bu}(t)|_{L^{\infty}}|\nabla_{\bv} \tilde{f}(t)|_{L^2_{\omega}} |\overline{f}(t)|_{L^2_{\omega}}.
 \end{aligned}
\end{equation}
By the Sobolev inequality \eqref{eq-sobol-infty}, solving the above Gronwall inequality gives rise to
\begin{equation} \label{eq-dfrwt-esti}
 \begin{aligned}
  \sup_{0\le t \le T}|f(t)-\tilde{f}(t)|_{L^2_{\omega}}\le & \left(|f_0-\tilde{f}_0|_{L^2_{\omega}}+\sup_{0\le t \le T}|\bu (t)-\tilde{\bu}(t)|_{D^1 \cap D^2}\int_0^T C|\nabla_{\bv}\tilde{f}(t)|_{L^2_{\omega}}dt\right)\\
    & \times \exp \left( \int_0^T C(1+|\bu(t)|_{D^1 \cap D^2}+R(t)|\tilde{f}(t)|_{H_{\omega}^1})dt \right),
  \end{aligned}
\end{equation}
where we have used the inequality
\[
 |f(t)(1+\bv^2)^{\frac12}|_{L^1} \le C|f(t)|_{L^2_{\omega}},
\]
similarly as in \eqref{eq-ini-f-lonenm}. This completes the proof.
\end{proof}
\vskip 0.3cm
\noindent \textit{Proof of Proposition \ref{prop-kine-cs-wp}}. We first mollify $f_0(\bx, \bv)$ and $\bu(t, \bx)$ by covolution, i.e.,
\[
 f_0^{\e}(\bx, \bv)=f_0*j_{\e}(\bx, \bv) \quad \text{and} \quad \bu^{\e}(t, \bx)=\bu * j_{\e}(t, \bx),
\]
where $j_{\e}$ is the standard mollifier. Using the contraction principle, one can prove
\begin{equation} \label{eq-kine-cs-appro}
     \begin{dcases}
         f_t^{\e} + \bv \cdot \nabla_{\bx} f^{\e}+ \nabla_{\bv} \cdot (L[f^{\e}]f^{\e}+(\bu^{\e}-\bv)f^{\e})=0,\\
         f^{\e}|_{t=0}=f_0^{\e}(\bx,\bv),
     \end{dcases}
\end{equation}
admits a unique local smooth solution by standard procedure. Combining with the a priori estimates Lemma \ref{lm-kine-cs-apriori} (1)-(2), one can extend the local smooth solution to be in the whole interval $[0,T]$. Continue to apply Lemma \ref{lm-kine-cs-apriori} (1)-(2) to $f^{\e_i}$. It follows from Lemma \ref{lm-kine-cs-apriori} (3) that
\[
 \sup_{0\le t \le T}|f^{\e_i}(t)-f^{\e_j}(t)|_{L_{\omega}^2} \le \exp(C(T))\left(|f_0^{\e_i}-f_0^{\e_j}|_{L_{\omega}^2} +C(T)\sup_{0\le t \le T}|\bu^{\e_i}(t)-\bu^{\e_j}(t)|_{D^1 \cap D^2} \right),
\]
where
\begin{multline}
  C(T)\le C \int_0^T (1+|\bu(t)|_{D^1 \cap D^2})dt \\+CT\Big(1+R(T)\Big)|f_0|_{H_{\omega}^1}\exp \left( C\int_0^T(1+R(t)+|\bu(t)|_{L^{\infty}}+|\nabla\bu(\tau)|_{L^{\infty}})d t \right).
\end{multline}
Thus, there exists $f(t, \bx, \bv) \in C([0,T];L_{\omega}^2(\bbr^3 \times \bbr^3))$ such that
\begin{equation}\label{eq-fcon-wtl}
 f^{\e_i}(t, \bx, \bv) \to f(t, \bx, \bv) \quad \text{in $C([0,T];L_{\omega}^2(\bbr^3 \times \bbr^3))$, as $\e_i \to 0$}.
\end{equation}
It is easy to see that Proposition \ref{prop-kine-cs-wp} (1) holds, and that the non-negative $f$ satisfies
\[
  f_t + \bv \cdot \nabla_{\bx} f+ \nabla_{\bv} \cdot (L[f]f+(\bu-\bv)f)=0 \quad \text{in $\mathcal{D}'((0,T)\times \bbr^3 \times \bbr^3)$}.
\]
Next we prove $f(t,\bx, \bv) \in C([0,T];H^1_{\omega}(\bbr^3 \times \bbr^3))$. Using Lemma \ref{lm-kine-cs-apriori} (2) for $f^{\e_i}(t, \bx, \bv)$, one infers by \eqref{eq-fcon-wtl} that there exists a subsequence, still denoted by $\{f^{\e_i}(t, \bx, \bv)\}$, such that
\begin{equation}\label{eq-fwekcon-wtsobo}
 \omega^{\frac12}\nabla_{\bx}f^{\e_i} \rightharpoonup \omega^{\frac12}\nabla_{\bx}f \quad \text{weakly-$\star$ in $L^{\infty}(0,T;L^2)$, as $\e_i \to 0$}.
\end{equation}
It is easy to show that
\begin{equation}\label{eq-wtfxone-distr}
  \begin{aligned}
    &(\omega^{\frac12}\nabla_{\bx}f)_t +\bv \cdot \nabla_{\bx}(\omega^{\frac12}\nabla_{\bx}f) +\nabla_{\bv} \cdot \Big(L[f]\otimes \omega^{\frac12}\nabla_{\bx}f +(\bu-\bv)\otimes \omega^{\frac12}\nabla_{\bx}f \Big)\\
    =&\bv \cdot \nabla_{\bx}\omega^{\frac12} \nabla_{\bx}f -\nabla_{\bx}L[f] \cdot \nabla_{\bv}f \omega^{\frac12} +L[f]\cdot \nabla_{\bv}\omega^{\frac12}\nabla_{\bx}f -\omega^{\frac12}f \nabla_{\bx} \nabla_{\bv} \cdot L[f]\\
    &-\omega^{\frac12}\nabla \bu \cdot \nabla_{\bv}f +(\bu-\bv)\cdot \nabla_{\bv}\omega^{\frac12}\nabla_{\bx}f \quad \text{in $\mathcal{D}'((0,T)\times \bbr^3 \times \bbr^3)$}.
  \end{aligned}
\end{equation}
From \eqref{eq-wtfxone-distr}, we infer that
\[
  (\omega^{\frac12}\nabla_{\bx}f)_t \in L^{\infty}(0,T;H^{-1}).
\]
This together with the fact that $\omega^{\frac12}\nabla_{\bx}f \in L^{\infty}(0,T;L^2)$ due to \eqref{eq-fwekcon-wtsobo}, implies
\begin{equation}\label{eq-wtfxone-wekcont}
  \omega^{\frac12}\nabla_{\bx}f \in C([0,T];L^2-W),
\end{equation}
which means that $\omega^{\frac12}\nabla_{\bx}f$ is continuous in $[0,T]$ with respect to the weak topology in $L^2(\bbr^3 \times \bbr^3)$.

Take $j_{\e}(\bx-\cdot, \bv-\cdot)$ as the test function and denote $f*j_{\e}$ by $\langle f \rangle_{\e}$. It follows from \eqref{eq-wtfxone-distr} that

\begin{equation} \label{eq-wtfxone-tes}
 \begin{aligned}
  &(\langle \omega^{\frac12}\nabla_{\bx} f\rangle_{\e})_t +\bv \cdot \nabla_{\bx} \langle\omega^{\frac12}\nabla_{\bx} f\rangle_{\e}+ \nabla_{\bv} \cdot \Big(L[f]\otimes \langle\omega^{\frac12}\nabla_{\bx} f\rangle_{\e} +(\bu-\bv)\otimes \langle\omega^{\frac12}\nabla_{\bx} f\rangle_{\e}\Big)\\
  =&\langle\bv\cdot\nabla_{\bx}\omega^{\frac12}\nabla_{\bx}f\rangle_{\e} -\langle\nabla_{\bx}L[f]\cdot\nabla_{\bv}f\omega^{\frac12}\rangle_{\e}\\
  &+\langle L[f]\cdot\nabla_{\bv}\omega^{\frac12}\nabla_{\bx}f\rangle_{\e}
   -\langle f\omega^{\frac12}\nabla_{\bx}\nabla_{\bv}\cdot L[f]\rangle_{\e}\\
  &-\langle \omega^{\frac12}\nabla \bu \cdot \nabla_{\bv}f\rangle_{\e} +\langle (\bu-\bv)\cdot\nabla_{\bv}\omega^{\frac12}\nabla_{\bx}f \rangle_{\e}\\
  &+\bv \cdot \nabla_{\bx}\langle \omega^{\frac12}\nabla_{\bx} f\rangle_{\e} -\nabla_{\bx}\cdot \langle \bv \otimes \omega^{\frac12}\nabla_{\bx} f\rangle_{\e}\\
  &+\nabla_{\bv} \cdot (L[f]\otimes \langle\omega^{\frac12}\nabla_{\bx} f\rangle_{\e}) -\nabla_{\bv}\cdot \langle L[f]\otimes \omega^{\frac12}\nabla_{\bx} f\rangle_{\e}\\
  &+\nabla_{\bv} \cdot ((\bu-\bv)\otimes \langle\omega^{\frac12}\nabla_{\bx} f\rangle_{\e}) -\nabla_{\bv}\cdot \langle (\bu-\bv)\otimes \omega^{\frac12}\nabla_{\bx} f\rangle_{\e}.
 \end{aligned}
\end{equation}
Multiplying \eqref{eq-wtfxone-tes} by $2\langle \omega^{\frac12}\nabla_{\bx} f\rangle_{\e}$, we obtain
\begin{equation} \label{eq-wtfxone-tes-sq}
 \begin{aligned}
  &(\langle \omega^{\frac12}\nabla_{\bx} f\rangle_{\e}^2)_t +\bv \cdot \nabla_{\bx} \langle\omega^{\frac12}\nabla_{\bx} f\rangle_{\e}^2+ \nabla_{\bv} \cdot \Big(L[f] \langle\omega^{\frac12}\nabla_{\bx} f\rangle_{\e}^2 +(\bu-\bv) \langle\omega^{\frac12}\nabla_{\bx} f\rangle_{\e}^2\Big)\\
  =&-\nabla_{\bv} \cdot L[f]\langle \omega^{\frac12}\nabla_{\bx} f\rangle_{\e}^2 +3\langle \omega^{\frac12}\nabla_{\bx} f\rangle_{\e}^2\\
  &+2\langle \omega^{\frac12}\nabla_{\bx} f\rangle_{\e} \cdot\langle\bv\cdot\nabla_{\bx}\omega^{\frac12}\nabla_{\bx}f -\langle\nabla_{\bx}L[f]\cdot\nabla_{\bv}f\omega^{\frac12}\rangle_{\e}\\
  &+2\langle \omega^{\frac12}\nabla_{\bx} f\rangle_{\e}
  \cdot \langle L[f]\cdot\nabla_{\bv}\omega^{\frac12}\nabla_{\bx}f
   - f\omega^{\frac12}\nabla_{\bx}\nabla_{\bv}\cdot L[f]\\
   &\qquad \qquad \qquad \quad-\omega^{\frac12}\nabla \bu \cdot \nabla_{\bv}f +(\bu-\bv)\cdot\nabla_{\bv}\omega^{\frac12}\nabla_{\bx}f \rangle_{\e}\\
  &+2\langle \omega^{\frac12}\nabla_{\bx} f\rangle_{\e}
  \cdot\bigg[\bv \cdot \nabla_{\bx}\langle \omega^{\frac12}\nabla_{\bx} f\rangle_{\e} -\nabla_{\bx}\cdot \langle \bv \otimes \omega^{\frac12}\nabla_{\bx} f\rangle_{\e}\bigg]\\
  &+2\langle \omega^{\frac12}\nabla_{\bx} f\rangle_{\e}
  \cdot\bigg[\nabla_{\bv} \cdot (L[f]\otimes \langle\omega^{\frac12}\nabla_{\bx} f\rangle_{\e}) -\nabla_{\bv}\cdot \langle L[f]\otimes \omega^{\frac12}\nabla_{\bx} f\rangle_{\e}\bigg]\\
  &+2\langle \omega^{\frac12}\nabla_{\bx} f\rangle_{\e}
  \cdot\bigg[\nabla_{\bv} \cdot ((\bu-\bv)\otimes \langle\omega^{\frac12}\nabla_{\bx} f\rangle_{\e}) -\nabla_{\bv}\cdot \langle (\bu-\bv)\otimes \omega^{\frac12}\nabla_{\bx} f\rangle_{\e}\bigg].
 \end{aligned}
\end{equation}
Integrating \eqref{eq-wtfxone-tes-sq} over $\bbr^3\times \bbr^3$ leads to
\begin{equation} \label{eq-wtfxone-tes-sqint}
 \begin{aligned}
   &\frac{d}{dt}|\langle \omega^{\frac12}\nabla_{\bx} f\rangle_{\e}|_{L^2}^{2}\\
   =&\int_{\bbr^6}\bigg(-\nabla_{\bv} \cdot L[f]\langle \omega^{\frac12}\nabla_{\bx} f\rangle_{\e}^2 +3\langle \omega^{\frac12}\nabla_{\bx} f\rangle_{\e}^2\bigg)d\bx d\bv \\
  &+\int_{\bbr^6}2\langle \omega^{\frac12}\nabla_{\bx} f\rangle_{\e} \cdot\langle\bv\cdot\nabla_{\bx}\omega^{\frac12}\nabla_{\bx}f -\langle\nabla_{\bx}L[f]\cdot\nabla_{\bv}f\omega^{\frac12}\rangle_{\e}d\bx d\bv\\
  &+\int_{\bbr^6}2\langle \omega^{\frac12}\nabla_{\bx} f\rangle_{\e}
  \cdot \langle L[f]\cdot\nabla_{\bv}\omega^{\frac12}\nabla_{\bx}f
   - f\omega^{\frac12}\nabla_{\bx}\nabla_{\bv}\cdot L[f]\\
   &\qquad \qquad \qquad \qquad-\omega^{\frac12}\nabla \bu \cdot \nabla_{\bv}f +(\bu-\bv)\cdot\nabla_{\bv}\omega^{\frac12}\nabla_{\bx}f \rangle_{\e}d\bx d\bv\\
  &+\int_{\bbr^6}2\langle \omega^{\frac12}\nabla_{\bx} f\rangle_{\e}
  \cdot\bigg[\bv \cdot \nabla_{\bx}\langle \omega^{\frac12}\nabla_{\bx} f\rangle_{\e} -\nabla_{\bx}\cdot \langle \bv \otimes \omega^{\frac12}\nabla_{\bx} f\rangle_{\e}\bigg]d\bx d\bv\\
  &+\int_{\bbr^6}2\langle \omega^{\frac12}\nabla_{\bx} f\rangle_{\e}
  \cdot\bigg[\nabla_{\bv} \cdot (L[f]\otimes \langle\omega^{\frac12}\nabla_{\bx} f\rangle_{\e}) -\nabla_{\bv}\cdot \langle L[f]\otimes \omega^{\frac12}\nabla_{\bx} f\rangle_{\e}\bigg]d\bx d\bv\\
  &+\int_{\bbr^6}2\langle \omega^{\frac12}\nabla_{\bx} f\rangle_{\e}
  \cdot\bigg[\nabla_{\bv} \cdot ((\bu-\bv)\otimes \langle\omega^{\frac12}\nabla_{\bx} f\rangle_{\e}) -\nabla_{\bv}\cdot \langle (\bu-\bv)\otimes \omega^{\frac12}\nabla_{\bx} f\rangle_{\e}\bigg]d\bx d\bv\\
   =:&\sum_{i=1}^{6}H_i.
 \end{aligned}
\end{equation}
We estimate each $H_i$ $(1\le i \le 6)$ as follows.
\begin{equation*}
 \begin{aligned}
  |H_1|=& \left|\int_{\bbr^6}\bigg(-\nabla_{\bv} \cdot L[f]\langle \omega^{\frac12}\nabla_{\bx} f\rangle_{\e}^2 +3\langle \omega^{\frac12}\nabla_{\bx} f\rangle_{\e}^2\bigg)d\bx d\bv \right|\\
    \le& C|f(t)|_{L^1}|\nabla_{\bx} f(t)|_{L^2_{\omega}}^2 +C|\nabla_{\bx} f(t)|_{L^2_{\omega}}^2;
\end{aligned}
\end{equation*}
\begin{equation*}
 \begin{aligned}
 |H_2|=& \left|\int_{\bbr^6}2\langle \omega^{\frac12}\nabla_{\bx} f\rangle_{\e} \cdot\langle\bv\cdot\nabla_{\bx}\omega^{\frac12}\nabla_{\bx}f -\langle\nabla_{\bx}L[f]\cdot\nabla_{\bv}f\omega^{\frac12}\rangle_{\e}d\bx d\bv \right|\\
 \le& C|\nabla_{\bx} f(t)|_{L^2_{\omega}}^2 +CR(t)|f(t)|_{L^1}|\nabla_{\bx} f(t)|_{L^2_{\omega}}|\nabla_{\bv} f(t)|_{L^2_{\omega}};
 \end{aligned}
\end{equation*}
\begin{equation*}
 \begin{aligned}
 |H_3|=&\Bigg|\int_{\bbr^6}2\langle \omega^{\frac12}\nabla_{\bx} f\rangle_{\e}
  \cdot \langle L[f]\cdot\nabla_{\bv}\omega^{\frac12}\nabla_{\bx}f
   - f\omega^{\frac12}\nabla_{\bx}\nabla_{\bv}\cdot L[f]\\
   &\qquad \qquad \qquad \qquad-\omega^{\frac12}\nabla \bu \cdot \nabla_{\bv}f +(\bu-\bv)\cdot\nabla_{\bv}\omega^{\frac12}\nabla_{\bx}f \rangle_{\e}d\bx d\bv \Bigg|\\
 \le& CR(t)|f(t)|_{L^1}|\nabla_{\bx} f(t)|_{L^2_{\omega}}^2 +C|f(t)|_{L^1}|\nabla_{\bx} f(t)|_{L^2_{\omega}}|f(t)|_{L^2_{\omega}}\\
 &+C|\nabla \bu(t)|_{L^{\infty}}|\nabla_{\bx} f(t)|_{L^2_{\omega}}|\nabla_{\bv} f(t)|_{L^2_{\omega}}
 +C|\bu(t)|_{L^{\infty}}|\nabla_{\bx} f(t)|_{L^2_{\omega}}^2 +C|\nabla_{\bx} f(t)|_{L^2_{\omega}}^2;
 \end{aligned}
\end{equation*}
\begin{equation*}
 \begin{aligned}
 |H_4|=&\left|\int_{\bbr^6}2\langle \omega^{\frac12}\nabla_{\bx} f\rangle_{\e}
  \cdot\bigg[\bv \cdot \nabla_{\bx}\langle \omega^{\frac12}\nabla_{\bx} f\rangle_{\e} -\nabla_{\bx}\cdot \langle \bv \otimes \omega^{\frac12}\nabla_{\bx} f\rangle_{\e}\bigg]d\bx d\bv \right|\\
  =&\Bigg|\int_{\bbr^6} \int_{\bbr^6} 2(\bv-\bw)\cdot \nabla_{\bx}j_{\e}(\bx-\bz,\bv-\bw)\omega^{\frac12}(\bz, \bw)\nabla_{\bx} f(t,\bz, \bw)d\bz d\bw\cdot \langle \omega^{\frac12}\nabla_{\bx} f\rangle_{\e}d\bx d\bv  \Bigg|\\
 \le&C|\nabla_{\bx} f(t)|_{L^2_{\omega}}^2,
 \end{aligned}
\end{equation*}
where we have used the facts that
\[
 |\bw-\bv|\le \e \quad \text{and}\quad |\nabla_{\bx}j_{\e}|_{L^1}\le \frac{C}{\e};
\]
\begin{equation*}
 \begin{aligned}
 |H_5|=&\left|\int_{\bbr^6} 2\langle \omega^{\frac12}\nabla_{\bx} f\rangle_{\e}
  \cdot\bigg[\nabla_{\bv} \cdot (L[f]\otimes \langle\omega^{\frac12}\nabla_{\bx} f\rangle_{\e}) -\nabla_{\bv}\cdot \langle L[f]\otimes \omega^{\frac12}\nabla_{\bx} f\rangle_{\e}\bigg] d\bx d\bv \right|\\
 \le&2\left|\int_{\bbr^6} \nabla_{\bv} \cdot L[f] \langle\omega^{\frac12}\nabla_{\bx} f\rangle_{\e}^2 d\bx d\bv\right|\\
 &+2\Bigg|\int_{\bbr^6} \int_{\bbr^6} \Big (L[f](t,\bx,\bv)-L[f](t,\bz,\bw)\Big)\cdot\nabla_{\bv}j_{\e}(\bx-\bz, \bv-\bw)\\
 &\qquad \qquad \qquad\Big(\omega^{\frac12}\nabla_{\bx}f\Big)(t, \bz, \bw)d\bz d\bw \cdot \langle \omega^{\frac12}\nabla_{\bx}f \rangle_{\e}d\bx d\bv \Bigg|\\
 \le& C|f(t)|_{L^1}|\nabla_{\bx} f(t)|_{L^2_{\omega}}^2 +C(1+R(t))|f(t)|_{L^1}|\nabla_{\bx} f(t)|_{L^2_{\omega}}^2,
 \end{aligned}
\end{equation*}
where we have used the facts that
\begin{equation*}
 \begin{aligned}
   &|L[f](t,\bx,\bv)-L[f](t,\bz,\bw)|\\
   =&\left|\int_{\bbr^6}\varphi (|\bx-\by|)f(t,\by,\bv^*)(\bv^*-\bv)d\by d\bv^* -\int_{\bbr^6}\varphi (|\bz-\by|)f(t,\by,\bv^*)(\bv^*-\bw)d\by d\bv^*\right|\\
   \le& CR(t)|f(t)|_{L^1}|\bx-\bz| +C|f(t)|_{L^1}|\bv-\bw|\\
   \le& C|f(t)|_{L^1}(1+R(t))\e
 \end{aligned}
\end{equation*}
and
\[
 |\nabla_{\bv}j_{\e}|_{L^1}\le \frac{C}{\e};
\]
\begin{equation*}
 \begin{aligned}
 |H_6|=&\left|\int_{\bbr^6} 2\langle \omega^{\frac12}\nabla_{\bx} f\rangle_{\e}
  \cdot\bigg[\nabla_{\bv} \cdot ((\bu-\bv)\otimes \langle\omega^{\frac12}\nabla_{\bx} f\rangle_{\e}) -\nabla_{\bv}\cdot \langle (\bu-\bv)\otimes \omega^{\frac12}\nabla_{\bx} f\rangle_{\e}\bigg] d\bx d\bv \right|\\
 \le& 2\Bigg|\int_{\bbr^6} \int_{\bbr^6} \Big [\bu(t,\bx)-\bv-(\bu(t,\bz)-\bw)\Big]\cdot\nabla_{\bv}j_{\e}(\bx-\bz, \bv-\bw)\\
 &\qquad \qquad \Big(\omega^{\frac12}\nabla_{\bx}f\Big)(t, \bz, \bw)d\bz d\bw \cdot \langle \omega^{\frac12}\nabla_{\bx}f \rangle_{\e}d\bx d\bv \Bigg|+ 6|\nabla_{\bx} f(t)|_{L^2_{\omega}}^2\\
 \le& C(1+|\nabla \bu(t)|_{L^{\infty}})|\nabla_{\bx} f(t)|_{L^2_{\omega}}^2.
 \end{aligned}
\end{equation*}
Using Lemma \ref{lm-kine-cs-apriori} (2) for $f^{\e_i}(t, \bx, \bv)$, we have
\[
 \sup_{0\le t \le T}|f^{\e_i}(t)|_{H^1_{\omega}} \le |f_0|_{H_{\omega}^1}\exp \left( C\int_0^T(1+R(t)+|\bu(t)|_{L^{\infty}}+|\nabla\bu(t)|_{L^{\infty}})d t\right)\quad \text{for all $i \in \bbn $}.
\]
Combining with \eqref{eq-fcon-wtl}, we infer that there exists a subsequence, still denoted by $\{f^{\e_i}(t, \bx, \bv)\}$, such that
\begin{equation}\label{eq-fcon-wk}
 f^{\e_i}(t, \bx, \bv) \rightharpoonup f(t, \bx, \bv) \quad \text{weakly-$\star$ in $L^{\infty}(0,T;H^1_{\omega}(\bbr^3 \times \bbr^3))$, as $\e_i \to 0$}.
\end{equation}
It follows from \eqref{eq-fcon-wk} that
\begin{equation}\label{eq-wtfsobo-nom}
 \begin{aligned}
  |f|_{L^{\infty}(0,T;H^1_{\omega})} \le& \liminf_{\e_i \to 0}|f^{\e_i}|_{L^{\infty}(0,T;H^1_{\omega})} \\ \le&|f_0|_{H_{\omega}^1}\exp \left( C\int_0^T(1+R(t)+|\bu(t)|_{L^{\infty}}+|\nabla\bu(t)|_{L^{\infty}})d t \right)\\
  \le& C(T).
 \end{aligned}
\end{equation}
From the assumption $\bu(t,\bx) \in C([0,T];D^1 \cap D^2)\cap L^2(0,T;D^{2,q})$, we know that there also exists $C(T)$ such that
\begin{equation}\label{eq-constu}
 \sup_{0\le t\le T}|\bu(t)|_{D^1\cap D^2} +\int_0^T|\bu(t)|_{D^{2,q}}^2 dt\le C(T).
\end{equation}
Substituting these estimates for $H_i$ $(1\le i \le 6)$ into \eqref{eq-wtfxone-tes-sqint} and integrating the resulting inequality over $[t_1, t_2]$, $\forall t_1, t_2 \in [0,T]$, we obtain by \eqref{eq-wtfsobo-nom} and \eqref{eq-constu} that
\[
\begin{aligned}
 \left| |\langle \omega^{\frac12}\nabla_{\bx} f(t_2) \rangle_{\e}|_{L^2}^2-|\langle^{\frac12}\nabla_{\bx} f(t_1) \rangle_{\e}|_{L^2}^2 \right|\le&  C(T)\int_{t_1}^{t_2}(1+|\nabla \bu(t)|_{L^{\infty}})dt\\
 \le&C(T)\Big(|t_2-t_1|+|t_2-t_1|^{\frac12}\Big),\quad \forall t_1, t_2 \in [0,T],
\end{aligned}
\]
Letting $\e \to 0$ yields
\begin{equation} \label{eq-fsoboconti}
  \left| | \nabla_{\bx}f(t_2)|_{L^2_{\omega}}^2-| \nabla_{\bx}f(t_1)|_{L^2_{\omega}}^2 \right|\le  C(T)\Big(|t_2-t_1|+|t_2-t_1|^{\frac12}\Big),\quad \forall t_1, t_2 \in [0,T].
\end{equation}
Combining \eqref{eq-wtfxone-wekcont} and \eqref{eq-fsoboconti}, we know
\[
 \nabla_{\bx}f \in C([0,T]; L^2_{\omega}).
\]
Similarly, we can prove $\nabla_{\bv}f \in C([0,T]; L^2_{\omega})$. Together with \eqref{eq-fcon-wtl}, it is shown that $f(t,\bx,\bv) \in C([0,T]; H^1_{\omega})$, and thus it is a strong solution to \eqref{eq-kine-cs}.

Assume $\tilde{f}$ is a strong solution to  \eqref{eq-kine-cs} with $\bu$ and $f_0$ replaced by $\tilde{\bu}$ and $\tilde{f}_0$, respectively. Similarly as the proof in Lemma \ref{lm-kine-cs-apriori} (3), one can demonstrate that
\begin{multline}\label{eq-kine-cs-stab}
 \sup_{0\le t \le T}|f(t)-\tilde{f}(t)|_{L_{\omega}^2} \le \left(|f_0-\tilde{f}_0|_{L^2_{\omega}}+\sup_{0\le t \le T}|\bu (t)-\tilde{\bu}(t)|_{D^1 \cap D^2}\int_0^T C|\nabla_{\bv}\tilde{f}(t)|_{L^2_{\omega}}dt\right)\\
     \times \exp \left( \int_0^T C(1+|\bu(t)|_{D^1 \cap D^2}+R(t)|\tilde{f}(t)|_{H_{\omega}^1})dt \right),
\end{multline}
which also implies uniqueness of the strong solution.  Similarly as \eqref{eq-wtfsobo-nom}, we infer by using the uniqueness and continuity of $f(t)$ in $H_{\omega}^1$ that
\[
|f(t)|_{H_{\omega}^1}\le |f_0|_{H_{\omega}^1}\exp \left( C\int_0^t(1+R(\tau)+|\bu(\tau)|_{L^{\infty}}+|\nabla\bu(\tau)|_{L^{\infty}})d \tau \right), \quad 0\le t \le T.
\]
This completes the proof. $\hfill \square$
%
%
\section{Local Existence of Strong Solutions to the Coupled System}\label{set-loc-exist}
\setcounter{equation}{0}
In this section, we prove the local existence of strong solutions to the coupled system \eqref{eq-cs-ns}. Our strategy is as follows. We first linearize the system and construct the approximate solutions by iteration; then we derive the uniform bound on the approximate solutions in a higher order norm for the short time; last we prove that the approximate solution sequence is the Cauchy sequence in a lower order norm, and further show that the limit is the desired local strong solution. Based on our analysis in Section \ref{sec-kine-cs}, we present the existence result for the linearized system without proof. The reader can refer to \cite{cho2004unique}\cite{cho2006existence} for details related to the linearized compressible Navier--Stokes equation.
\begin{proposition}\label{prop-cs-ns-line}
Assume $\hat{\bu}(t,\bx) \in C([0,T];D^1\cap D^{2}(\bbr^3))\cap L^2(0,T;D^{2,q}(\bbr^3))$, $3<q\le 6$, and $\hat\bu_t(t,\bx)\in L^2(0,T;D^{1}(\bbr^3))$. Under the initial conditions in Theorem \ref{thm-loc-exist}, the following linearized system
\begin{equation} \label{eq-cs-ns-line}
     \begin{dcases}
         f_t + \bv \cdot \nabla_{\bx} f+ \nabla_{\bv} \cdot (L[f]f+(\hat{\bu}-\bv)f)=0,\\
         \rho_t+\nabla \cdot (\rho \hat{\bu})=0,\\
         (\rho \bu)_t+\nabla \cdot (\rho \hat{\bu} \otimes \bu)+\nabla P=\mu \Delta \bu +(\mu+\lambda)\nabla \nabla \cdot \bu+\int_{\bbr^3}f(\bv-\bu)d\bv,
     \end{dcases}
\end{equation}
admits a unique strong solution $(f(t,\bx,\bv), \rho(t,\bx), \bu(t,\bx))$, satisfying
\[
\begin{aligned}
  &f(t,\bx,\bv)\in C([0,T];H_{\omega}^1(\bbr^3 \times \bbr^3)),\\
  &\rho(t,\bx)\in C([0,T];H^1\cap W^{1,q}(\bbr^3)),\\
  &\bu(t,\bx)\in C([0,T];D^1\cap D^{2}(\bbr^3))\cap L^2(0,T;D^{2,q}(\bbr^3)),\\
  &\rho_t(t,\bx)\in C([0,T];L^2\cap L^{q}(\bbr^3)),\quad \bu_t(t,\bx)\in L^2(0,T;D^{1}(\bbr^3)),\\
  &\rho^{\frac12}\bu_t(t,\bx)\in L^{\infty}(0,T;L^2(\bbr^3)).
\end{aligned}
\]
\end{proposition}
Next we use Proposition \ref{prop-cs-ns-line} to finish the proof of Theorem \ref{thm-loc-exist}.
\vskip 0.3cm
\noindent \textit{Proof of Theorem \ref{thm-loc-exist}}. We first construct the approximate solutions by iteration. Given $\bu^n(t, \bx) \in C([0,T];D^1\cap D^{2}(\bbr^3))\cap L^2(0,T;D^{2,q}(\bbr^3))$, $3<q\le 6$, and $\bu^n_t(t,\bx)\in L^2(0,T;D^{1}(\bbr^3))$, with $\bu^n|_{t=0}=\bu_0$ in $D^1\cap D^{2}(\bbr^3)$, $(f^{n+1},\rho^{n+1}, \bu^{n+1})$ is determined by
\begin{equation} \label{eq-cs-ns-appro}
     \begin{dcases}
         f^{n+1}_t + \bv \cdot \nabla_{\bx} f^{n+1}+ \nabla_{\bv} \cdot (L[f^{n+1}]f^{n+1}+(\bu^{n}-\bv)f^{n+1})=0,\\
         \rho^{n+1}_t+\nabla \cdot (\rho^{n+1} \bu^{n})=0,\\
         (\rho^{n+1} \bu^{n+1})_t+\nabla \cdot (\rho^{n+1} \bu^{n} \otimes \bu^{n+1})+\nabla P(\rho^{n+1})=\mu \Delta \bu^{n+1}\\ \qquad  +(\mu+\lambda)\nabla \nabla \cdot \bu^{n+1}+\int_{\bbr^3}f^{n+1}(\bv-\bu^{n+1})d\bv,
     \end{dcases}
\end{equation}
subject to the initial data
\[
 f^{n+1}|_{t=0}=f_0, \quad \rho^{n+1}|_{t=0}=\rho_0, \quad \bu^{n+1}|_{t=0}=\bu_0.
\]
Using Proposition \ref{prop-cs-ns-line}, we know $(f^{n+1},\rho^{n+1}, \bu^{n+1})$ is well-defined. In the iteration procedure, $\bu^0$ is set by
\begin{equation} \label{eq-cs-ns-approini}
     \begin{dcases}
       \bu^0_t=\Delta\bu^0,\\
       \bu^0|_{t=0}=\bu_0 \in D^1\cap D^2.
     \end{dcases}
\end{equation}
It is easy to see
\[
 \bu^0 \in C([0,T];D^1\cap D^{2})\cap L^2(0,T;D^{2,q}) \quad\text{and} \quad \bu^0_t \in L^2(0,T;D^{1}).
\]
Moreover, it holds that
\[
 \sup_{0\le t \le T}|\bu^0(t)|_{D^1\cap D^2}^2+ \int_0^T \Big(|\bu^0_t(t)|_{D^1}^2+|\bu^0(t)|_{D^{2,q}}^2\Big)dt \le C|\bu_0|_{D^1\cap D^2}^2.
\]
\vskip 0.3cm
\noindent \textbf{Uniform Bound in a Higher Order Norm}
\vskip 0.3cm
Define
\[
 C_0:=|f_0|_{H^1_{\omega}}+|\rho_0|_{H^1\cap W^{1,q}}+|\bu_0|_{D^1\cap D^2}+|g|_{L^2}+1.
\]
Suppose that there exists $T_* \in (0,T]$ such that
\begin{equation}\label{eq-indu-assup}
 \sup_{0\le t \le T_*}\Big(|\bu^n(t)|_{D^1}+\beta^{-1}|\bu^n(t)|_{D^2}\Big)+\int_0^{T_*} \Big(|\bu^n_t(t)|_{D^1}^2+|\bu^n(t)|_{D^{2,q}}^2\Big)dt \le C_1, \quad n\in \bbn,
\end{equation}
where $\beta$ and $C_1$ are to be determined later. Next we prove by induction that \eqref{eq-indu-assup} holds for all  $n\in \bbn$.

Multiplying $\eqref{eq-cs-ns-appro}_2$ by $r (\rho^{n+1})^{r-1}$, $2\le r \le 6$, we have
\begin{equation}\label{eq-appro-massr}
 \frac{\partial}{\partial t}(\rho^{n+1})^{r}+\nabla \cdot [(\rho^{n+1})^{r}\bu]=-(r-1)(\rho^{n+1})^{r}\nabla\cdot \bu^n.
\end{equation}
Integrating \eqref{eq-appro-massr} over $\bbr^3$ gives
\begin{equation}\label{eq-appro-massr-gron}
  \frac{d}{dt}|\rho^{n+1}(t)|_{L^r}^r \le C(r)|\nabla \bu^n(t)|_{L^{\infty}}|\rho^{n+1}(t)|_{L^r}^r.
\end{equation}
Applying $\nabla$ to $\eqref{eq-cs-ns-appro}_2$ leads to
\begin{equation}\label{eq-appro-massdifone}
(\nabla \rho^{n+1})_t +\nabla \bu^n\cdot \nabla \rho^{n+1} +\bu^n\cdot \nabla\nabla \rho^{n+1}+\nabla \rho^{n+1}\nabla\cdot \bu^n+\rho^{n+1}\nabla \nabla\cdot \bu^n=0.
\end{equation}
Multiplying \eqref{eq-appro-massdifone} by $r|\nabla \rho^{n+1}|^{r-2}\nabla \rho^{n+1}$, we obtain
\begin{equation}\label{eq-appro-massdifoner}
  \begin{gathered}
   (|\nabla \rho^{n+1}|^{r})_t +\nabla\cdot(|\nabla \rho^{n+1}|^{r}\bu^n)=-r|\nabla \rho^{n+1}|^{r-2}\nabla \rho^{n+1}\cdot \nabla\bu^n \cdot \nabla \rho^{n+1}\\-(r-1)|\nabla \rho^{n+1}|^{r}\nabla\cdot \bu^n -r\rho^{n+1}|\nabla \rho^{n+1}|^{r-2}\nabla \rho^{n+1}\cdot\nabla \nabla\cdot \bu^n.
  \end{gathered}
\end{equation}
Integrating \eqref{eq-appro-massdifoner} over $\bbr^3$ gives rise to
\begin{equation}\label{eq-appro-massdifoner-gron}
  \begin{aligned}
    \frac{d}{dt}|\nabla \rho^{n+1}(t)|^{r}_{L^r} \le& C(r)|\nabla\bu^n(t)|_{L^{\infty}}|\nabla \rho^{n+1}(t)|^{r}_{L^r} +C(r)|\rho^{n+1}(t)|_{L^p}|\nabla \rho^{n+1}(t)|^{r-1}_{L^r}|\nabla^2\bu^n(t)|_{L^q}\\
    \le&C(r)|\nabla\bu^n(t)|_{L^{\infty}}|\nabla \rho^{n+1}(t)|^{r}_{L^r} +C(r)|\rho^{n+1}(t)|_{W^{1,r}}|\nabla \rho^{n+1}(t)|^{r-1}_{L^r}|\nabla^2\bu^n(t)|_{L^q},
  \end{aligned}
\end{equation}
where $p$, $q$ and $r$ satisfy
\[
 \frac{1}{p}+\frac{r-1}{r}+\frac{1}{q}=1,
\]
and we have used the Sobolev inequality
\[
  |\rho^{n+1}(t)|_{L^p}\le C|\rho^{n+1}(t)|_{W^{1,r}}.
\]
Combining \eqref{eq-appro-massr-gron} and \eqref{eq-appro-massdifoner-gron}, we deduce that
\begin{equation}\label{eq-appro-masswoner-gron}
  \begin{aligned}
    \frac{d}{dt}|\rho^{n+1}(t)|_{W^{1,r}} \le& C(r)|\nabla\bu^n(t)|_{L^{\infty}}|\rho^{n+1}(t)|_{W^{1,r}} +C(r)|\nabla^2\bu^n(t)|_{L^q}|\nabla \rho^{n+1}(t)|_{W^{1,r}}\\
    \le&C(r)|\nabla\bu^n(t)|_{W^{1,q}}|\rho^{n+1}(t)|_{W^{1,r}}.
  \end{aligned}
\end{equation}
Solving the above Gronwall inequality gives
\begin{equation}\label{eq-appro-masswoner-norm}
  |\rho^{n+1}(t)|_{W^{1,r}} \le  |\rho_0|_{W^{1,r}}\exp\left(\int_0^t C(r)|\nabla\bu^n(\tau)|_{W^{1,q}}d\tau\right), \quad 0\le t \le T.
\end{equation}
It is easy to see that
\begin{equation}\label{eq-appro-rhot-norm}
  \begin{aligned}
    |\rho^{n+1}_t(t)|_{L^2\cap L^q}\le& |\bu^n(t)|_{L^{\infty}}|\nabla\rho^{n+1}(t)|_{L^2\cap L^q}+|\rho^{n+1}(t)|_{L^{\infty}}|\nabla\bu^{n}(t)|_{L^2\cap L^q}\\
    \le&C|\bu^n(t)|_{D^1\cap D^2}|\nabla\rho^{n+1}(t)|_{L^2\cap L^q}+C|\rho^{n+1}(t)|_{W^{1,q}}|\nabla\bu^{n}(t)|_{H^1}\\
    \le&C\Big(|\nabla\rho^{n+1}(t)|_{L^2\cap L^q}+|\rho^{n+1}(t)|_{W^{1,q}}\Big) |\nabla\bu^{n}(t)|_{H^1}.
  \end{aligned}
\end{equation}
Using the induction hypothesis \eqref{eq-indu-assup} and taking $T_1:=T_1(q,C_1)$ suitably small, we infer from \eqref{eq-appro-masswoner-norm}, \eqref{eq-appro-rhot-norm} and Proposition \ref{prop-kine-cs-wp} (1)-(2) that
\begin{equation}\label{eq-appro-rhoandrhot-norm}
  \begin{aligned}
    &\sup_{0\le t \le T_1}|\rho^{n+1}(t)|_{H^1 \cap W^{1,q}}\le C|\rho_0|_{H^1\cap W^{1,q}} \le C C_0,\\
    &\sup_{0\le t \le T_1}|\rho^{n+1}_t(t)|_{L^2 \cap L^q}\le C(1+\beta)C_0 C_1,\\
    &\sup_{0\le t \le T_1}|f^{n+1}(t)|_{H^1_{\omega}}\le C|f_0|_{H^1_{\omega}}\le C C_0.
  \end{aligned}
\end{equation}
Differentiating $\eqref{eq-cs-ns-appro}_3$ with respect to $t$, we deduce that
\begin{equation}\label{eq-appro-momdift}
  \begin{gathered}
  \Big(\rho^{n+1} \bu^{n+1}_t\Big)_t+ \Big(\rho^{n+1} \bu^{n} \cdot\nabla \bu^{n+1}\Big)_t+\Big(\nabla P(\rho^{n+1})\Big)_t=\mu \Delta \bu^{n+1}_t +(\mu+\lambda)\nabla \nabla \cdot \bu^{n+1}_t\\ +\Bigg(\int_{\bbr^3}f^{n+1}(\bv-\bu^{n+1})d\bv\Bigg)_t, \quad \text{in $\mathcal{D}'((0,T)\times \bbr^3\times \bbr^3)$}.
  \end{gathered}
\end{equation}
Take $\bu^{n+1}_t$ as the test function. It follows from \eqref{eq-appro-momdift} that
\begin{equation}\label{eq-appro-dif-ns-ener}
  \begin{aligned}
    &\frac12 \frac{d}{dt}\left|\sqrt{\rho^{n+1}} \bu^{n+1}_t\right|_{L^2}^2 +\mu\left|\nabla \bu^{n+1}_t\right|_{L^2}^2 +(\mu+\lambda)\left|\nabla \cdot \bu^{n+1}_t\right|_{L^2}^2\\
    =&-\int_{\bbr^3}2\rho^{n+1} \bu^{n} \cdot\nabla \bu^{n+1}_t\cdot \bu^{n+1}_t d\bx-\int_{\bbr^3}\rho^{n+1}_t \bu^{n} \cdot\nabla \bu^{n+1}\cdot \bu^{n+1}_t d\bx\\
    &-\int_{\bbr^3}\rho^{n+1} \bu^{n}_t \cdot\nabla \bu^{n+1}\cdot \bu^{n+1}_t d\bx-\int_{\bbr^3}\Big(\nabla P(\rho^{n+1})\Big)_t \cdot \bu^{n+1}_t d\bx\\
    &+\int_{\bbr^3}\Bigg(\int_{\bbr^3}f^{n+1}(\bv-\bu^{n+1})d\bv\Bigg)_t \cdot \bu^{n+1}_t d\bx\\
    =:&\sum_{i=1}^5 J_i.
  \end{aligned}
\end{equation}
We estimate each $J_i$ $(1\le i\le 5)$ as follows.
\begin{equation*}
  \begin{aligned}
    |J_1|=&\left|\int_{\bbr^3}2\rho^{n+1} \bu^{n} \cdot\nabla \bu^{n+1}_t\cdot \bu^{n+1}_t d\bx\right|\\
    \le&2|\rho^{n+1}|_{L^{\infty}}^{\frac12}|\bu^{n}|_{L^{\infty}}\left|\sqrt{\rho^{n+1}} \bu^{n+1}_t\right|_{L^2}\left|\nabla \bu^{n+1}_t\right|_{L^2}\\
    \le& \frac{\mu}{14}\left|\nabla \bu^{n+1}_t\right|_{L^2}^2+ C|\rho^{n+1}|_{L^{\infty}}|\bu^{n}|_{L^{\infty}}^2 \left|\sqrt{\rho^{n+1}} \bu^{n+1}_t\right|_{L^2}^2;\\[0.4cm]
    |J_2|=&\left|\int_{\bbr^3}\rho^{n+1}_t \bu^{n} \cdot\nabla \bu^{n+1}\cdot \bu^{n+1}_t d\bx\right|\\
    \le&|\bu^{n}|_{L^{\infty}}|\rho^{n+1}_t|_{L^3}\left|\nabla \bu^{n+1}\right|_{L^2}\left|\nabla \bu^{n+1}_t\right|_{L^2}\\
    \le& \frac{\mu}{14}\left|\nabla \bu^{n+1}_t\right|_{L^2}^2+ C|\bu^{n}|_{L^{\infty}}^2|\rho^{n+1}_t|_{L^3}^2
    \left|\nabla \bu^{n+1}\right|_{L^2}^2;\\[0.4cm]
     |J_3|=&\left|\int_{\bbr^3}\rho^{n+1} \bu^{n}_t \cdot\nabla \bu^{n+1}\cdot \bu^{n+1}_t d\bx\right|\\
    \le&|\rho^{n+1}|_{L^{\infty}}^{\frac34}|\bu^n_t|_{L^6}\left|\nabla \bu^{n+1}\right|_{L^2}\left|\sqrt{\rho^{n+1}} \bu^{n+1}_t\right|_{L^2}^{\frac12}\left|\nabla\bu^{n+1}_t\right|_{L^2}^{\frac12}\\
    \le& \frac{\mu}{14}\left|\nabla \bu^{n+1}_t\right|_{L^2}^2+ \eta |\nabla \bu^n_t|_{L^2}^2\left|\nabla \bu^{n+1}\right|_{L^2}^2 +C(\eta)|\rho^{n+1}|_{L^{\infty}}^3 \left|\sqrt{\rho^{n+1}} \bu^{n+1}_t\right|_{L^2}^2;\\[0.4cm]
    |J_4|=&\left|\int_{\bbr^3}\Big(\nabla P(\rho^{n+1})\Big)_t \cdot \bu^{n+1}_t d\bx\right|\\
    \le&|\Big(P(\rho^{n+1})\Big)_t|_{L^2}\left|\nabla \bu^{n+1}_t\right|_{L^2}\\
    \le& \frac{\mu}{14}\left|\nabla \bu^{n+1}_t\right|_{L^2}^2 +C(P')|\rho^{n+1}_t|_{L^2}^2;\\[4mm]
    J_5=&\int_{\bbr^3}\int_{\bbr^3}f^{n+1}_t\Big(\bv-\bu^{n+1}\Big)d\bv\cdot \bu^{n+1}_t d\bx-\int_{\bbr^3}\int_{\bbr^3}f^{n+1}d\bv \Big|\bu^{n+1}_t\Big|^2 d\bx\\
    \le&\int_{\bbr^3}\int_{\bbr^3}f^{n+1} \bv\otimes\bv d\bv : \nabla \bu^{n+1}_t d\bx-\int_{\bbr^3}\int_{\bbr^3}f^{n+1} \bv d\bv \cdot \nabla \Big(\bu^{n+1}\cdot\bu^{n+1}_t\Big) d\bx\\
    &+\int_{\bbr^3}\int_{\bbr^3}\Big(L[f^{n+1}]f^{n+1}+(\bu^n-\bv)f^{n+1}\Big)  d\bv \cdot \bu^{n+1}_t d\bx\\
    \le&\left|\int_{\bbr^3}f^{n+1} \bv^2 d\bv\right|_{L^2}\left|\nabla\bu^{n+1}_t\right|_{L^2}+2\left|\int_{\bbr^3}f^{n+1} \bv d\bv\right|_{L^3}\left|\nabla\bu^{n+1}\right|_{L^2}\left|\nabla\bu^{n+1}_t\right|_{L^2}\\
    &+C\left|f^{n+1}(1+\bv^2)^{\frac12}\right|_{L^1}\left|\int_{\bbr^3}f^{n+1}(1+\bv^2)^{\frac12}d\bv \right|_{L^{\frac65}}\left|\nabla\bu^{n+1}_t\right|_{L^2}\\
    &+|\bu^n|_{L^{\infty}}\left|\int_{\bbr^3}f^{n+1}d\bv \right|_{L^{\frac65}}\left|\nabla\bu^{n+1}_t\right|_{L^2}+\left|\int_{\bbr^3}f^{n+1}\bv d\bv \right|_{L^{\frac65}}\left|\nabla\bu^{n+1}_t\right|_{L^2}\\
    \le& \frac{\mu}{14}\left|\nabla \bu^{n+1}_t\right|_{L^2}^2+ C|f^{n+1}|_{H_{\omega}^1}^2\left|\nabla \bu^{n+1}\right|_{L^2}^2+ C\Big(1+|\bu^n|_{L^{\infty}}^2\Big)|f^{n+1}|_{L_{\omega}^2}^2+ C|f^{n+1}|_{L_{\omega}^2}^4,
  \end{aligned}
\end{equation*}
In the estimate of $J_5$, we have used the following inequalities.
\begin{equation*}
  \begin{aligned}
    \left|\int_{\bbr^3}f^{n+1} \bv^2 d\bv\right|_{L^2}=&\left|\int_{\bbr^3}f^{n+1} \bv^2 (1+\bv^2)^{\frac{\alpha}{4}}(1+\bv^2)^{-\frac{\alpha}{4}}d\bv\right|_{L^2}\\
    \le& \left|(1+\bv^2)^{-\frac{\alpha}{4}}\right|_{L^2}\left| f^{n+1} \bv^2 (1+\bv^2)^{\frac{\alpha}{4}}\right|_{L^2}\\
    \le& C|f^{n+1}|_{L_{\omega}^2};
  \end{aligned}
\end{equation*}
\begin{equation*}
  \begin{aligned}
    \left| f^{n+1} (1+\bv^2)^{\frac{1}{2}}\right|_{L^1}=&\left|\int_{\bbr^3}\int_{\bbr^3}f^{n+1} (1+\bv^2)^{\frac{1}{2}}(1+\bx^2+\bv^2)^{\frac{\alpha}{2}}(1+\bx^2+\bv^2)^{-\frac{\alpha}{2}}d\bx d\bv\right|\\
    \le& \left|(1+\bx^2+\bv^2)^{-\frac{\alpha}{2}}\right|_{L^2}|f^{n+1}|_{L_{\omega}^2}\\
    \le& C|f^{n+1}|_{L_{\omega}^2};
  \end{aligned}
\end{equation*}
\begin{equation*}
  \begin{aligned}
    \left|\int_{\bbr^3}f^{n+1} \bv d\bv\right|_{L^3}=&\left|\int_{\bbr^3}f^{n+1} \bv d\bv\right|_{L^1}^{\frac15}\left|\int_{\bbr^3}f^{n+1} \bv d\bv\right|_{L^6}^{\frac45}\\
    \le&\frac15 \left|\int_{\bbr^3}f^{n+1} \bv d\bv\right|_{L^1}+ \frac45 \left|\int_{\bbr^3}|\nabla f^{n+1}| |\bv| d\bv\right|_{L^2}\\
    \le &C|f^{n+1}|_{H_{\omega}^1};\\[3mm]
    \left|\int_{\bbr^3} f^{n+1} (1+\bv^2)^{\frac{1}{2}}d\bv\right|_{L^{\frac65}}\le &\left|\int_{\bbr^3}f^{n+1} (1+\bv^2)^{\frac{1}{2}}d\bv\right|_{L^1}^{\frac23}\left|\int_{\bbr^3}f^{n+1} (1+\bv^2)^{\frac{1}{2}} d\bv\right|_{L^2}^{\frac13}\\
    \le&\frac23 \left|\int_{\bbr^3}f^{n+1} (1+\bv^2)^{\frac{1}{2}} d\bv\right|_{L^1}+ \frac13 \left|\int_{\bbr^3} f^{n+1}(1+\bv^2)^{\frac{1}{2}} d\bv\right|_{L^2}\\
    \le &C|f^{n+1}|_{L_{\omega}^2}.
  \end{aligned}
\end{equation*}
Substituting the estimates on $J_i$ $(1\le i \le 5)$ into \eqref{eq-appro-dif-ns-ener}, we obtain
\begin{equation}\label{eq-appro-dif-ns-en-gron}
  \begin{aligned}
    &\frac12 \frac{d}{dt}\left|\sqrt{\rho^{n+1}} \bu^{n+1}_t\right|_{L^2}^2 +\frac{9}{14}\mu\left|\nabla \bu^{n+1}_t\right|_{L^2}^2\\
    \le&\bigg(C|\rho^{n+1}|_{L^{\infty}}|\bu^{n}|_{L^{\infty}}^2+C(\eta)|\rho^{n+1}|_{L^{\infty}}^3 \bigg)\left|\sqrt{\rho^{n+1}} \bu^{n+1}_t\right|_{L^2}^2\\
    &+\bigg(\eta |\nabla \bu^n_t|_{L^2}^2 +C|\bu^{n}|_{L^{\infty}}^2|\rho^{n+1}_t|_{L^3}^2+ C|f^{n+1}|_{H_{\omega}^1}^2\bigg)\left|\nabla \bu^{n+1}\right|_{L^2}^2\\
    &+C(P')|\rho^{n+1}_t|_{L^2}^2 +C\Big(1+|\bu^n|_{L^{\infty}}^2\Big)|f^{n+1}|_{L_{\omega}^2}^2+ C|f^{n+1}|_{L_{\omega}^2}^4.
  \end{aligned}
\end{equation}
Multiplying $\eqref{eq-cs-ns-appro}_3$ by $\bu^{n+1}_t$, and integrating the resulting equation over $\bbr^3$ lead to
\begin{equation}\label{eq-app-gut}
  \begin{aligned}
    &\frac{d}{dt}\left(\frac{\mu}{2}\left|\nabla \bu^{n+1}\right|_{L^2}^2+\frac{\mu+\lambda}{2}\left|\nabla \cdot \bu^{n+1}\right|_{L^2}^2\right)+\left|\sqrt{\rho^{n+1}} \bu^{n+1}_t\right|_{L^2}^2\\
    =&-\int_{\bbr^3}\rho^{n+1}\bu^{n}\cdot\nabla \bu^{n+1}\cdot \bu^{n+1}_t d\bx -\int_{\bbr^3}\nabla P\Big(\rho^{n+1}\Big)\cdot \bu^{n+1}_t d\bx\\
    &+\int_{\bbr^3}\int_{\bbr^3}f^{n+1}(\bv- \bu^{n+1})d\bv \cdot \bu^{n+1}_t d\bx\\
    =:&\sum_{i=1}^{3}K_i.
  \end{aligned}
\end{equation}
We estimate each $K_i$ $(1\le i \le 3)$ as follows.
\begin{equation*}
  \begin{aligned}
    |K_1|=&\left|\int_{\bbr^3}\rho^{n+1}\bu^{n}\cdot\nabla \bu^{n+1}\cdot \bu^{n+1}_t d\bx \right|\\
    \le&|\rho^{n+1}|_{L^{\infty}}^{\frac12}\left|\sqrt{\rho^{n+1}} \bu^{n+1}_t\right|_{L^2}\left|\nabla \bu^{n+1}\right|_{L^2}|\bu^n|_{L^{\infty}}\\
    \le& \frac12 \left|\sqrt{\rho^{n+1}} \bu^{n+1}_t\right|_{L^2}^2+ C|\rho^{n+1}|_{L^{\infty}}|\bu^n|_{L^{\infty}}^2\left|\nabla \bu^{n+1}\right|_{L^2}^2;\\[4mm]
    |K_2|=&\left|\int_{\bbr^3}\nabla P(\rho^{n+1})\cdot \bu^{n+1}_t d\bx \right|\\
    \le& |P(\rho^{n+1})|_2 \left|\nabla \bu^{n+1}_t\right|_{L^2}\\
    \le& \frac{\mu}{14}\left|\nabla \bu^{n+1}_t\right|_{L^2}^2+ C(P')|\rho^{n+1}|_{L^2}^2;\\[4mm]
    |K_3|=&\left|\int_{\bbr^3}\int_{\bbr^3}f^{n+1}(\bv- \bu^{n+1})d\bv \cdot \bu^{n+1}_t d\bx \right|\\
    \le& \left|\int_{\bbr^3}f^{n+1}\bv d\bv \right|_{L^{\frac65}}\left|\nabla \bu^{n+1}_t\right|_{L^2}+ \left|\int_{\bbr^3}f^{n+1} d\bv \right|_{L^{\frac32}}\left|\nabla \bu^{n+1}\right|_{L^2}\left|\nabla \bu^{n+1}_t\right|_{L^2}\\
    \le& \frac{\mu}{14}\left|\nabla \bu^{n+1}_t\right|_{L^2}^2+  C|f^{n+1}|_{L_{\omega}^2}^2\left|\nabla \bu^{n+1}\right|_{L^2}^2+ C|f^{n+1}|_{L_{\omega}^2}^2.
  \end{aligned}
\end{equation*}
Here in the estimate of $K_3$, we have used the inequality
\[
\begin{aligned}
 \left|\int_{\bbr^3} f^{n+1} \bv d\bv\right|_{L^{\frac32}}\le &\left|\int_{\bbr^3}f^{n+1} (1+\bv^2)^{\frac{1}{2}}d\bv\right|_{L^1}^{\frac13}\left|\int_{\bbr^3}f^{n+1} (1+\bv^2)^{\frac{1}{2}} d\bv\right|_{L^2}^{\frac23}\\
    \le&\frac13 \left|\int_{\bbr^3}f^{n+1} (1+\bv^2)^{\frac{1}{2}} d\bv\right|_{L^1}+ \frac23 \left|\int_{\bbr^3} f^{n+1}(1+\bv^2)^{\frac{1}{2}} d\bv\right|_{L^2}\\
    \le &C|f^{n+1}|_{L_{\omega}^2}.
\end{aligned}
\]
Substituting these estimates into \eqref{eq-app-gut}, we arrive at
\begin{equation}\label{eq-app-gut-gron}
  \begin{aligned}
    &\frac{d}{dt}\left(\frac{\mu}{2}\left|\nabla \bu^{n+1}\right|_{L^2}^2+\frac{\mu+\lambda}{2}\left|\nabla \cdot \bu^{n+1}\right|_{L^2}^2\right)+\frac12\left|\sqrt{\rho^{n+1}} \bu^{n+1}_t\right|_{L^2}^2\\
    \le&\bigg(C|\rho^{n+1}|_{L^{\infty}}|\bu^n|_{L^{\infty}}^2+  C|f^{n+1}|_{L_{\omega}^2}^2\bigg)\left|\nabla \bu^{n+1}\right|_{L^2}^2 +\frac{\mu}{7}\left|\nabla \bu^{n+1}_t\right|_{L^2}^2\\
    &+ C(P')|\rho^{n+1}|_{L^2}^2 +C|f^{n+1}|_{L_{\omega}^2}^2.
  \end{aligned}
\end{equation}
Combining \eqref{eq-appro-dif-ns-en-gron} with \eqref{eq-app-gut-gron}, we deduce that
\begin{equation}\label{eq-appro-en-gv-gron}
  \begin{aligned}
    &\frac{d}{dt}\left(\frac12 \left|\sqrt{\rho^{n+1}} \bu^{n+1}_t\right|_{L^2}^2 +\frac{\mu}{2}\left|\nabla \bu^{n+1}\right|_{L^2}^2+\frac{\mu+\lambda}{2}\left|\nabla \cdot \bu^{n+1}\right|_{L^2}^2\right)+\frac{\mu}{2}\left|\nabla \bu^{n+1}_t\right|_{L^2}^2\\
    \le&\bigg(C|\rho^{n+1}|_{L^{\infty}}|\bu^{n}|_{L^{\infty}}^2+C(\eta)|\rho^{n+1}|_{L^{\infty}}^3 \bigg)\left|\sqrt{\rho^{n+1}} \bu^{n+1}_t\right|_{L^2}^2\\
    &+\bigg(\eta |\nabla \bu^n_t|_{L^2}^2 +C|\bu^{n}|_{L^{\infty}}^2|\rho^{n+1}_t|_{L^3}^2+ C|f^{n+1}|_{H_{\omega}^1}^2+C|\rho^{n+1}|_{L^{\infty}}|\bu^n|_{L^{\infty}}^2+  C|f^{n+1}|_{L_{\omega}^2}^2\bigg)\left|\nabla \bu^{n+1}\right|_{L^2}^2\\
    &+C(P')|\rho^{n+1}_t|_{L^2}^2 +C(P')|\rho^{n+1}|_{L^2}^2 +C\Big(1+|\bu^n|_{L^{\infty}}^2\Big)|f^{n+1}|_{L_{\omega}^2}^2+ C|f^{n+1}|_{L_{\omega}^2}^4.
  \end{aligned}
\end{equation}
By the compatibility condition, we infer that
\[
 \limsup_{t \to 0}\left|\sqrt{\rho^{n+1}} \bu^{n+1}_t(t)\right|_{L^2}\le |g|_{L^2}+|\rho_0|_{L^{\infty}}^{\frac12}|\bu_0|_{L^{\infty}}|\nabla \bu_0|_{L^2}\le CC_0^{\frac52}.
\]
Take $\eta$ and then $T_2:=T_2(\eta, \beta, C_0, C_1)$ suitably small. We obtain by solving the Gronwall inequality \eqref{eq-appro-en-gv-gron} that
\begin{equation}\label{eq-app-en-gu-bd}
  \sup_{0\le t\le T_2}\left(\left|\sqrt{\rho^{n+1}} \bu^{n+1}_t(t)\right|_{L^2}^2 +\left|\nabla \bu^{n+1}(t)\right|_{L^2}^2\right)+\int_0^{T_2} \left|\nabla \bu^{n+1}_t(t)\right|_{L^2}^2dt \le CC_0^5.
\end{equation}
Using the elliptic estimate, we have
\begin{equation}\label{eq-app-usec-nm}
  \begin{aligned}
    |\bu^{n+1}|_{D^2}\le& C\bigg(|\rho^{n+1} \bu^{n+1}_t|_{L^2} +|\rho^{n+1}\bu^n\cdot \nabla\bu^{n+1}|_{L^2}\\&+|\nabla P(\rho^{n+1})|_{L^2}+\left|\int_{\bbr^3}f^{n+1}(\bv- \bu^{n+1})d\bv \right|_{L^2} \bigg)\\
    \le& C|\rho^{n+1}|_{L^{\infty}}^{\frac12}\left|\sqrt{\rho^{n+1}} \bu^{n+1}_t\right|_{L^2} +C|\rho^{n+1}|_{L^{\infty}}|\bu^n|_{L^{\infty}}\left|\nabla \bu^{n+1}\right|_{L^2}\\ &+C(P')|\nabla\rho^{n+1}|_{L^2}
    +C|f^{n+1}|_{L_{\omega}^2}+C|f^{n+1}|_{H_{\omega}^1}\left|\nabla \bu^{n+1}\right|_{L^2}\\
    \le& CC_0^{\frac72} +CC_0^{\gamma}+C\beta^{\frac12}C_0^{\frac72}C_1\\
    \le& C \beta^{\frac12}C_0^{\gamma+\frac72}C_1;
  \end{aligned}
\end{equation}
\begin{equation}\label{eq-app-usecsix-nm}
  \begin{aligned}
    |\bu^{n+1}|_{D^{2,6}}\le& C\bigg(|\rho^{n+1} \bu^{n+1}_t|_{L^6} +|\rho^{n+1}\bu^n\cdot \nabla\bu^{n+1}|_{L^6}\\&+|\nabla P(\rho^{n+1})|_{L^6}+\left|\int_{\bbr^3}f^{n+1}(\bv- \bu^{n+1})d\bv \right|_{L^6} \bigg)\\
    \le& C|\rho^{n+1}|_{L^{\infty}}\left|\nabla \bu^{n+1}_t\right|_{L^2} +C|\rho^{n+1}|_{L^{\infty}}|\bu^n|_{L^{\infty}}\left|\nabla \bu^{n+1}\right|_{H^1}\\ &+C(P')|\nabla\rho^{n+1}|_{L^6}
    +C|\nabla_{\bx}f^{n+1}|_{L_{\omega}^2}+C|\nabla_{\bx}f^{n+1}|_{L_{\omega}^2}|\bu^{n+1}|_{L^{\infty}}\\
    \le& C|\rho^{n+1}|_{L^{\infty}}\left|\nabla \bu^{n+1}_t\right|_{L^2} +C|\rho^{n+1}|_{L^{\infty}}|\bu^n|_{D^1\cap D^2}\left|\nabla \bu^{n+1}\right|_{H^1}\\ &+C(P')|\nabla\rho^{n+1}|_{L^6}
    +C|\nabla_{\bx}f^{n+1}|_{L_{\omega}^2}+C|\nabla_{\bx}f^{n+1}|_{L_{\omega}^2}|\bu^{n+1}|_{D^1\cap D^2}.
  \end{aligned}
\end{equation}
By interpolation, we know
\[
  |\bu^{n+1}|_{D^{2,q}}^2 \le C\Big( |\bu^{n+1}|_{D^{2}}^2+|\bu^{n+1}|_{D^{2,6}}^2\Big), \quad 3< q \le 6.
\]
Take $T_3:=T_3(\beta, C_0, C_1)$ suitably small and $T_3\le T_2$. We deduce by using \eqref{eq-app-en-gu-bd} that
\begin{equation}\label{eq-app-gusecq-bd}
 \begin{aligned}
   \int_0^{T_3} |\bu^{n+1}_t(t)|_{D^{2,q}}^2 dt \le& \sup_{0\le t\le T_3}|\rho^{n+1}(t)|_{L^{\infty}}^2\int_0^{T_3} \left|\nabla \bu^{n+1}_t(t)\right|_{L^2}^2dt + CC_0^7\\
   \le& CC_0^7.
  \end{aligned}
\end{equation}
Set
\[
  T_*:=\min \{T_1, T_2, T_3\}, \quad \beta:=CC_0^{2\gamma+7}, \quad C_1:=CC_0^7.
\]
Combining \eqref{eq-app-en-gu-bd}, \eqref{eq-app-usec-nm} and \eqref{eq-app-gusecq-bd}, we obtain
\begin{equation}\label{eq-indu-next}
 \sup_{0\le t \le T_*}\Big(|\bu^{n+1}(t)|_{D^1}+\beta^{-1}|\bu^{n+1}(t)|_{D^2}\Big)+\int_0^{T_*} \Big(|\bu^{n+1}_t(t)|_{D^1}^2+|\bu^{n+1}(t)|_{D^{2,q}}^2\Big)dt \le C_1.
\end{equation}
It is apparent that $\bu^0$ satisfy the induction assumption \eqref{eq-indu-assup}. Thus, by induction, we conclude that \eqref{eq-indu-next} holds for all $n \in \bbn$.
\vskip 3mm
\noindent \textbf{Convergence in a Lower Order Norm}
\vskip 3mm
Define
\[
 \overline{f}^{n+1}:=f^{n+1}-f^n,\quad \overline{\rho}^{n+1}:=\rho^{n+1}-\rho^n,
\]
\[
 \overline{\bu}^{n+1}:=\bu^{n+1}-\bu^n,\quad \overline{P}^{n+1}:=P(\rho^{n+1})-P(\rho^n).
\]
It follows from $\eqref{eq-cs-ns-appro}_3$ that
\begin{equation}\label{eq-app-nsmt-dif}
  \begin{aligned}
    &\rho^{n+1}\overline{\bu}^{n+1}_t +\rho^{n+1}\bu^n\cdot\nabla\overline{\bu}^{n+1} +\nabla\overline{P}^{n+1}\\
    =&\mu\Delta\overline{\bu}^{n+1} +(\mu+\lambda)\nabla\nabla\cdot\overline{\bu}^{n+1} -\int_{\bbr^3}f^{n+1}\overline{\bu}^{n+1}d\bv \\
    &-\overline{\rho}^{n+1}\bu^n_t-\Big(\rho^{n+1}\overline{\bu}^n +\overline{\rho}^{n+1}\bu^{n-1}\Big)\cdot\nabla\bu^n +\int_{\bbr^3}\overline{f}^{n+1}(\bv-\bu^n)d\bv
  \end{aligned}
\end{equation}
Multiplying \eqref{eq-app-nsmt-dif} by $\overline{\bu}^{n+1}$, and integrating the resulting equation over $\bbr^3$ lead to
\begin{equation}\label{eq-app-nsmt-dif-en}
  \begin{aligned}
    &\frac12\frac{d}{dt}\left| \sqrt{\rho^{n+1}}\overline{\bu}^{n+1}\right|_{L^2}^2 +\mu\left|\nabla\overline{\bu}^{n+1}\right|_{L^2}^2 +(\mu+\lambda)\left|\nabla\cdot\overline{\bu}^{n+1}\right|_{L^2}^2\\
    =& -\int_{\bbr^3}\nabla\overline{P}^{n+1}\cdot\overline{\bu}^{n+1} d\bx -\int_{\bbr^3}\int_{\bbr^3}f^{n+1}d\bv|\overline{\bu}^{n+1}|^2d\bx\\
    &-\int_{\bbr^3}\overline{\rho}^{n+1}\bu^n_t\cdot\overline{\bu}^{n+1}d\bx -\int_{\bbr^3}\rho^{n+1}\overline{\bu}^n\cdot\nabla\bu^n\cdot\overline{\bu}^{n+1}d\bx \\
    &-\int_{\bbr^3}\overline{\rho}^{n+1}\bu^{n-1}\cdot\nabla\bu^n\cdot\overline{\bu}^{n+1}d\bx  +\int_{\bbr^3}\int_{\bbr^3}\overline{f}^{n+1}(\bv-\bu^n)d\bv\cdot\overline{\bu}^{n+1}d\bx\\
    =:&\sum_{i=1}^{6}M_i.
  \end{aligned}
\end{equation}
We estimate each $M_i$ $(1 \le i \le 6)$ as follows.
\begin{equation*}
  \begin{aligned}
    |M_1|=& \left|\int_{\bbr^3}\nabla\overline{P}^{n+1}\cdot\overline{\bu}^{n+1} d\bx\right|\\
    \le& \left|\overline{P}^{n+1}\right|_{L^2}\left|\nabla\overline{\bu}^{n+1}\right|_{L^2} \le \frac{\mu}{8}\left|\nabla\overline{\bu}^{n+1}\right|_{L^2}^2 +C(P')\left|\overline{\rho}^{n+1}\right|_{L^2}^2;\\[4mm]
    M_2=&-\int_{\bbr^3}\int_{\bbr^3}f^{n+1}d\bv|\overline{\bu}^{n+1}|^2d\bx\le 0;\\[4mm]
    |M_3|=& \left|\int_{\bbr^3}\overline{\rho}^{n+1}\bu^n_t\cdot\overline{\bu}^{n+1}d\bx\right|\\
    \le& \left|\overline{\rho}^{n+1}\right|_{L^{\frac32}} |\nabla\bu^n_t|_{L^2} \left|\nabla\overline{\bu}^{n+1}\right|_{L^2}\\
    \le& \frac{\mu}{8}\left|\nabla\overline{\bu}^{n+1}\right|_{L^2}^2 +C|\nabla\bu^n_t|_{L^2}^2 \left|\overline{\rho}^{n+1}\right|_{L^{\frac32}}^2;\\[4mm]
    |M_4|=& \left|\int_{\bbr^3}\rho^{n+1}\overline{\bu}^n\cdot\nabla\bu^n\cdot\overline{\bu}^{n+1}d\bx\right|\\
    \le& |\rho^{n+1}|_{L^{\infty}}^{\frac12} \left| \sqrt{\rho^{n+1}}\overline{\bu}^{n+1}\right|_{L^2} \left|\nabla\overline{\bu}^n\right|_{L^2} |\nabla\bu^n|_{L^3}\\
    \le& \e \left|\nabla\overline{\bu}^n\right|_{L^2}^2 +C(\e)|\rho^{n+1}|_{L^{\infty}} |\nabla\bu^n|_{L^3}^2 \left| \sqrt{\rho^{n+1}}\overline{\bu}^{n+1}\right|_{L^2}^2;\\[4mm]
    |M_5|=& \left|\int_{\bbr^3}\overline{\rho}^{n+1}\bu^{n-1}\cdot\nabla\bu^n\cdot\overline{\bu}^{n+1}d\bx \right|\\
    \le& \left| \overline{\rho}^{n+1}\right|_{L^2}|\bu^{n-1}|_{L^{\infty}}|\nabla\bu^n|_{L^3}  \left|\nabla\overline{\bu}^{n+1}\right|_{L^2} \\
    \le& \frac{\mu}{8}\left|\nabla\overline{\bu}^{n+1}\right|_{L^2}^2 +C|\bu^{n-1}|_{L^{\infty}}^2|\nabla\bu^n|_{L^3}^2 \left|\overline{\rho}^{n+1}\right|_{L^2}^2;\\[4mm]
    |M_6|=&\left|\int_{\bbr^3}\int_{\bbr^3}\overline{f}^{n+1}(\bv-\bu^n)d\bv\cdot\overline{\bu}^{n+1}d\bx \right|\\
    \le& \frac{\mu}{8}\left|\nabla\overline{\bu}^{n+1}\right|_{L^2}^2+ C\left|\int_{\bbr^3}\overline{f}^{n+1}\bv d\bv\right|_{L^{\frac65}}^2 +C|\bu^n|_{L^{\infty}}^2\left|\int_{\bbr^3}\overline{f}^{n+1} d\bv\right|_{L^{\frac65}}^2\\
    \le& \frac{\mu}{8}\left|\nabla\overline{\bu}^{n+1}\right|_{L^2}^2 +C\Big(1+|\bu^n|_{L^{\infty}}^2\Big)\left|\overline{f}^{n+1} \left(1+ \bv^2\right)^{\frac{1+\beta}{2}}\right|_{L^{\frac65}}^2.
  \end{aligned}
\end{equation*}
In the estimate of $M_6$, we have used the following inequality
\[
\begin{aligned}
  &\left|\int_{\bbr^3}\Big|\overline{f}^{n+1}\Big|(1+ \bv^2)^{\frac12}d\bv\right|_{L^{\frac65}}\\
  \le& \left(\int_{\bbr^3}\bigg[\Big|\overline{f}^{n+1}\Big| \left(1+ \bv^2\right)^{\frac{1+\beta}{2}}\bigg]^{\frac65}d\bv d\bx \right)^{\frac56} \left(\int_{\bbr^3}(1+ \bv^2)^{-3\beta}d\bv \right)^{\frac16}\\
  \le& C\left|\overline{f}^{n+1} \left(1+ \bv^2\right)^{\frac{1+\beta}{2}}\right|_{L^{\frac65}}, \quad \beta>\frac12.
 \end{aligned}
\]
Define
\[
  \Lambda(\bv):= \left(1+ \bv^2\right)^{\frac{1+\beta}{2}}, \quad \beta>\frac12.
\]
Substituting these estimates into \eqref{eq-app-nsmt-dif-en}, we obtain
\begin{equation}\label{eq-app-nsmt-dif-engron}
  \begin{aligned}
    &\frac{d}{dt}\left| \sqrt{\rho^{n+1}}\overline{\bu}^{n+1}\right|_{L^2}^2 +\mu\left|\nabla\overline{\bu}^{n+1}\right|_{L^2}^2\\
    \le& C(\e)|\rho^{n+1}|_{L^{\infty}} |\nabla\bu^n|_{L^3}^2 \left| \sqrt{\rho^{n+1}}\overline{\bu}^{n+1}\right|_{L^2}^2 +\Big(C(P')+C|\bu^{n-1}|_{L^{\infty}}^2|\nabla\bu^n|_{L^3}^2\Big) \left|\overline{\rho}^{n+1}\right|_{L^2}^2\\
    &+C|\nabla\bu^n_t|_{L^2}^2 \left|\overline{\rho}^{n+1}\right|_{L^{\frac32}}^2 +C\Big(1+|\bu^n|_{L^{\infty}}^2\Big)\left|\overline{f}^{n+1}\Lambda\right|_{L^{\frac65}}^2 +\e \left|\nabla\overline{\bu}^n\right|_{L^2}^2.
  \end{aligned}
\end{equation}
It follows from $\eqref{eq-cs-ns-appro}_2$ that
\begin{equation}\label{eq-app-mas-dif}
  \overline{\rho}^{n+1}_t +\bu^n\cdot\nabla\overline{\rho}^{n+1} +\overline{\rho}^{n+1}\nabla\cdot\bu^n +\overline{\bu}^n\cdot\nabla\rho^n +\rho^n\nabla\cdot\overline{\bu}^n=0.
\end{equation}
Multiplying \eqref{eq-app-mas-dif} by $2 \overline{\rho}^{n+1}$, and integrating the resulting equation over $\bbr^3$, we obtain
\begin{equation}\label{eq-app-mas-dif-squ}
  \begin{aligned}
    \frac{d}{dt}\left| \overline{\rho}^{n+1}\right|_{L^2}^2=&\int_{\bbr^3}\left| \overline{\rho}^{n+1}\right|^2 \nabla \cdot \bu^n d\bx +\int_{\bbr^3}2 \overline{\rho}^{n+1}\overline{\bu}^n \cdot \nabla\rho^n d\bx +\int_{\bbr^3}2 \overline{\rho}^{n+1}\rho^n\nabla\cdot\overline{\bu}^n d\bx\\
    \le&\Big(|\nabla \bu^n|_{L^{\infty}} +C(\e)|\nabla\rho^n|_{L^3}^2 +C(\e)|\rho^n|_{L^{\infty}}^2\Big)\left| \overline{\rho}^{n+1}\right|_{L^2}^2 +\e \left|\nabla\overline{\bu}^n\right|_{L^2}^2.
  \end{aligned}
\end{equation}
Multiplying \eqref{eq-app-mas-dif} by $\frac32 \left|\overline{\rho}^{n+1}\right|^{\frac12}\text{sgn}\overline{\rho}^{n+1}$ gives
\begin{equation}\label{eq-app-mas-dif-frc}
  \begin{gathered}
    \bigg(\left|\overline{\rho}^{n+1}\right|^{\frac32}\bigg)_t +\nabla \cdot\bigg(\left|\overline{\rho}^{n+1}\right|^{\frac32}\bu^n\bigg) +\frac12 \left|\overline{\rho}^{n+1}\right|^{\frac32}\nabla\cdot\bu^n\\
    +\frac32 \left|\overline{\rho}^{n+1}\right|^{\frac12}\text{sgn}\overline{\rho}^{n+1}\overline{\bu}^n \cdot\nabla\rho^n +\frac32 \left|\overline{\rho}^{n+1}\right|^{\frac12}\text{sgn}\overline{\rho}^{n+1}\rho^n \nabla\cdot \overline{\bu}^n=0.
  \end{gathered}
\end{equation}
Integrating \eqref{eq-app-mas-dif-frc} over $\bbr^3$ leads to
\begin{equation}\label{eq-app-mas-dif-frcint}
  \frac{d}{dt}\left|\overline{\rho}^{n+1}\right|_{L^{\frac32}}^{\frac32}\le |\nabla \bu^n|_{L^{\infty}}\left|\overline{\rho}^{n+1}\right|_{L^{\frac32}}^{\frac32} +C\left|\overline{\rho}^{n+1}\right|_{L^{\frac32}}^{\frac12}\left|\nabla\overline{\bu}^n\right|_{L^2} |\nabla\rho^n|_{L^2}
\end{equation}
We deduce from \eqref{eq-app-mas-dif-frcint} that
\begin{equation}\label{eq-app-mas-dif-frcint-gron}
  \frac{d}{dt}\left|\overline{\rho}^{n+1}\right|_{L^{\frac32}}^2\le \Big(2|\nabla \bu^n|_{L^{\infty}} +C(\e)|\nabla\rho^n|_{L^2}^2\Big)\left|\overline{\rho}^{n+1}\right|_{L^{\frac32}}^2 +\e \left|\nabla\overline{\bu}^n\right|_{L^2}^2.
\end{equation}
It follows from $\eqref{eq-cs-ns-appro}_1$ that
\begin{equation}\label{eq-app-kcs-dif-frc}
  \begin{gathered}
    \Big(\overline{f}^{n+1}\Big)_t +\bv \cdot\nabla_{\bx} \overline{f}^{n+1} +\nabla_{\bv}\cdot\Big[L[f^{n+1}]\overline{f}^{n+1}+(\bu^n-\bv)\overline{f}^{n+1}\Big]\\
      +\nabla_{\bv}\cdot\Big[L[\overline{f}^{n+1}]f^n +f^n \overline{\bu}^n\Big]=0.
  \end{gathered}
\end{equation}
Recall that $\Lambda(\bv)= \left(1+ \bv^2\right)^{\frac{1+\beta}{2}}$, $\ \beta>\frac12$. Multiplying \eqref{eq-app-kcs-dif-frc} by $\Lambda(\bv)$, we deduce that
\begin{equation}\label{eq-app-kcswt-dif}
  \begin{aligned}
    &\Big(\overline{f}^{n+1}\Lambda\Big)_t +\bv \cdot\nabla_{\bx} \Big(\overline{f}^{n+1}\Lambda\Big) +\nabla_{\bv}\cdot\Big[L[f^{n+1}]\overline{f}^{n+1}\Lambda+(\bu^n-\bv)\overline{f}^{n+1}\Lambda\Big]\\
    =&L[f^{n+1}]\cdot\nabla_{\bv}\Lambda\overline{f}^{n+1} +(\bu^n-\bv)\cdot\nabla_{\bv}\Lambda\overline{f}^{n+1}\\
      &-\bigg(\nabla_{\bv}\cdot L[\overline{f}^{n+1}]f^n +L[\overline{f}^{n+1}]\cdot\nabla_{\bv} f^n \bigg)\Lambda -\overline{\bu}^n \cdot\nabla_{\bv} f^n \Lambda.
  \end{aligned}
\end{equation}
Multiplying \eqref{eq-app-kcswt-dif} by $\frac65 \left|\overline{f}^{n+1}\Lambda\right|^{\frac15}\text{sgn}\overline{f}^{n+1}$ leads to
\begin{equation}\label{eq-app-kcswt-dif-frac}
  \begin{aligned}
    &\frac{\partial}{\partial t}\left|\overline{f}^{n+1}\Lambda\right|^{\frac65} +\bv \cdot\nabla_{\bx} \left|\overline{f}^{n+1}\Lambda\right|^{\frac65} +\nabla_{\bv}\cdot\bigg[L[f^{n+1}]\left|\overline{f}^{n+1}\Lambda\right|^{\frac65} +(\bu^n-\bv)\left|\overline{f}^{n+1}\Lambda\right|^{\frac65}\bigg]\\
    =&-\frac15 \nabla\cdot L[f^{n+1}]\left|\overline{f}^{n+1}\Lambda\right|^{\frac65} +\frac35 \left|\overline{f}^{n+1}\Lambda\right|^{\frac65}\\
    &+\frac65 \left|\overline{f}^{n+1}\Lambda\right|^{\frac15}L[f^{n+1}]\cdot\nabla_{\bv}\Lambda \left|\overline{f}^{n+1}\right| +\frac65 \left|\overline{f}^{n+1}\Lambda\right|^{\frac15}(\bu^n-\bv)\cdot\nabla_{\bv}\Lambda \left|\overline{f}^{n+1}\right|\\
    &-\frac65 \left|\overline{f}^{n+1}\Lambda\right|^{\frac15}\text{sgn}\overline{f}^{n+1} \bigg(\nabla_{\bv}\cdot L[\overline{f}^{n+1}]f^n +L[\overline{f}^{n+1}]\cdot\nabla_{\bv} f^n \bigg)\Lambda\\ &-\frac65 \left|\overline{f}^{n+1}\Lambda\right|^{\frac15}\text{sgn}\overline{f}^{n+1}\overline{\bu}^n \cdot\nabla_{\bv} f^n \Lambda.
  \end{aligned}
\end{equation}
Integrating \eqref{eq-app-kcswt-dif-frac} over $\bbr^3\times \bbr^3$ gives
\begin{equation}\label{eq-app-kcswt-dif-frac-gron}
  \begin{aligned}
    &\frac{d}{d t}\left|\overline{f}^{n+1}\Lambda\right|_{L^{\frac65}}^{\frac65}\\
    =&\int_{\bbr^3}\int_{\bbr^3}\Bigg(-\frac15 \nabla\cdot L[f^{n+1}]\left|\overline{f}^{n+1}\Lambda\right|^{\frac65} +\frac35 \left|\overline{f}^{n+1}\Lambda\right|^{\frac65}\Bigg)d\bx d\bv \\
    &+\int_{\bbr^3}\int_{\bbr^3}\Bigg(\frac65 \left|\overline{f}^{n+1}\Lambda\right|^{\frac15}L[f^{n+1}]\cdot\nabla_{\bv}\Lambda \left|\overline{f}^{n+1}\right| +\frac65 \left|\overline{f}^{n+1}\Lambda\right|^{\frac15}(\bu^n-\bv)\cdot\nabla_{\bv}\Lambda \left|\overline{f}^{n+1}\right|\Bigg)d\bx d\bv\\
    &-\int_{\bbr^3}\int_{\bbr^3} \frac65 \left|\overline{f}^{n+1}\Lambda\right|^{\frac15}\text{sgn}\overline{f}^{n+1} \bigg(\nabla_{\bv}\cdot L[\overline{f}^{n+1}]f^n +L[\overline{f}^{n+1}]\cdot\nabla_{\bv} f^n \bigg)\Lambda d\bx d\bv\\ &-\int_{\bbr^3}\int_{\bbr^3} \frac65 \left|\overline{f}^{n+1}\Lambda\right|^{\frac15}\text{sgn}\overline{f}^{n+1}\overline{\bu}^n \cdot\nabla_{\bv} f^n \Lambda d\bx d\bv\\
    =:&\sum_{i=1}^4N_i.
  \end{aligned}
\end{equation}
We estimate each $N_i$ $(1\le i \le 4)$ as follows.
\begin{equation*}
  \begin{aligned}
     |N_1|\le& C|f^{n+1}|_{L^1}\left|\overline{f}^{n+1}\Lambda\right|_{L^{\frac65}}^{\frac65} +C\left|\overline{f}^{n+1}\Lambda\right|_{L^{\frac65}}^{\frac65};\\[3mm]
     |N_2|\le& C|f^{n+1}(1+\bv^2)^{\frac12}|_{L^1}\left|\overline{f}^{n+1}\Lambda\right|_{L^{\frac65}}^{\frac65} +C|\bu^n|_{L^{\infty}}\left|\overline{f}^{n+1}\Lambda\right|_{L^{\frac65}}^{\frac65};\\[3mm]
     |N_3|\le& C\left|\overline{f}^{n+1}\Lambda\right|_{L^{\frac65}}^{\frac15} \left|\overline{f}^{n+1}\right|_{L^1} |f^n\Lambda|_{L^{\frac65}} +C\left|\overline{f}^{n+1}\Lambda\right|_{L^{\frac65}}^{\frac15} \left|\overline{f}^{n+1}(1+\bv^2)^{\frac12}\right|_{L^1} |(1+\bv^2)^{\frac12}\nabla_{\bv}f^n\Lambda|_{L^{\frac65}}\\
     \le& C\bigg(|f^n\Lambda|_{L^{\frac65}} +|(1+\bv^2)^{\frac12}\nabla_{\bv}f^n\Lambda|_{L^{\frac65}}\bigg) \left|\overline{f}^{n+1}\Lambda\right|_{L^{\frac65}}^{\frac15} \left|\overline{f}^{n+1}(1+\bv^2)^{\frac12}\right|_{L^1}\\
     \le& C|f^n|_{H_{\omega}^1} \left|\overline{f}^{n+1}\Lambda\right|_{L^{\frac65}}^{\frac15} \left|\overline{f}^{n+1}(1+\bv^2)^{\frac12}\right|_{L^1};\\[3mm]
     |N_4|\le& C\left|\overline{f}^{n+1}\Lambda\right|_{L^{\frac65}}^{\frac15} \left|\nabla\overline{\bu}^n\right|_{L^2} |\nabla_{\bv}f^n\Lambda|_{L^{\frac32}}\\
     \le& C|\nabla_{\bv}f^n|_{L_{\omega}^2}\left|\overline{f}^{n+1}\Lambda\right|_{L^{\frac65}}^{\frac15} \left|\nabla\overline{\bu}^n\right|_{L^2}.
  \end{aligned}
\end{equation*}
In the estimates of $N_3$ and $N_4$, we have used the following inequalities.
\begin{equation*}
  \begin{aligned}
     |f^n\Lambda|_{L^{\frac65}}\le& \Bigg(\int_{\bbr^3}\int_{\bbr^3}|f^n|^2(1+\bv^2)^{1+\beta}(1+\bx^2+\bv^2)^{\frac{2}{3}\alpha} d\bx d\bv \Bigg)^{\frac12}|(1+\bx^2+\bv^2)^{-\frac{\alpha}{3}}|_{L^3}\\
     \le& C|f^n|_{L_{\omega}^2};
  \end{aligned}
\end{equation*}
\begin{equation*}
  \begin{aligned}
     &|(1+\bv^2)^{\frac12}\nabla_{\bv}f^n\Lambda|_{L^{\frac65}}\\
     \le& \Bigg(\int_{\bbr^3}\int_{\bbr^3}|\nabla_{\bv}f^n|^2(1+\bv^2)^{2+\beta} (1+\bx^2+\bv^2)^{\frac{2}{3}\alpha} d\bx d\bv \Bigg)^{\frac12}|(1+\bx^2+\bv^2)^{-\frac{\alpha}{3}}|_{L^3}\\
     \le& C|\nabla_{\bv}f^n|_{L_{\omega}^2};
   \end{aligned}
\end{equation*}
\begin{equation*}
  \begin{aligned}
      |\nabla_{\bv}f^n\Lambda|_{L^{\frac32}}
     \le& \Bigg(\int_{\bbr^3}\int_{\bbr^3}|\nabla_{\bv}f^n|^2(1+\bv^2)^{1+\beta} (1+\bx^2+\bv^2)^{\frac{\alpha}{3}} d\bx d\bv \Bigg)^{\frac12}|(1+\bx^2+\bv^2)^{-\frac{\alpha}{6}}|_{L^6}\\
     \le& C|\nabla_{\bv}f^n|_{L_{\omega}^2}.
  \end{aligned}
\end{equation*}
Substituting these estimates into \eqref{eq-app-kcswt-dif-frac-gron}, we deduce that
\begin{equation}\label{eq-app-kcswt-dif-fracsqu-gron}
  \begin{aligned}
    \frac{d}{d t}\left|\overline{f}^{n+1}\Lambda\right|_{L^{\frac65}}^2
    \le& \bigg(C +C|f^{n+1}(1+\bv^2)^{\frac12}|_{L^1} +C|\bu^n|_{L^{\infty}} +C(\e)|\nabla_{\bv}f^n|_{L_{\omega}^2}^2\bigg)\left|\overline{f}^{n+1}\Lambda\right|_{L^{\frac65}}^2\\
    &+\bigg(C|f^n|_{L_{\omega}^2}^2 +C|\nabla_{\bv}f^n|_{L_{\omega}^2}^2\bigg)\left|\overline{f}^{n+1}(1+\bv^2)^{\frac12}\right|_{L^1}^2 +\e \left|\nabla\overline{\bu}^n\right|_{L^2}^2.
  \end{aligned}
\end{equation}
Similarly, we have
\begin{equation}\label{eq-app-kcswt-dif-lonesqu-gron}
  \begin{aligned}
    \frac{d}{d t}\left|\overline{f}^{n+1}(1+\bv^2)^{\frac12}\right|_{L^1}^2
    \le& \bigg(C +C|f^{n+1}(1+\bv^2)^{\frac12}|_{L^1} +C|\bu^n|_{L^{\infty}} +C|\nabla_{\bv}f^n|_{L_{\omega}^2}\\ &\quad +C(\e)|\nabla_{\bv}f^n|_{L_{\omega}^2}^2\bigg)\left|\overline{f}^{n+1}(1+\bv^2)^{\frac12}\right|_{L^1}^2 +\e \left|\nabla\overline{\bu}^n\right|_{L^2}^2.
  \end{aligned}
\end{equation}
Define
\[
 F^{n+1}(t):=\left| \sqrt{\rho^{n+1}}\overline{\bu}^{n+1}\right|_{L^2}^2 +\left| \overline{\rho}^{n+1}\right|_{L^2}^2 +\left|\overline{\rho}^{n+1}\right|_{L^{\frac32}}^2 +\left|\overline{f}^{n+1}\Lambda\right|_{L^{\frac65}}^2 +\left|\overline{f}^{n+1}(1+\bv^2)^{\frac12}\right|_{L^1}^2.
\]
Combining \eqref{eq-app-nsmt-dif-engron}, \eqref{eq-app-mas-dif-squ}, \eqref{eq-app-mas-dif-frcint-gron}, \eqref{eq-app-kcswt-dif-fracsqu-gron} and \eqref{eq-app-kcswt-dif-lonesqu-gron} , we obtain
\begin{equation}\label{eq-app-dif-gron}
  \begin{aligned}
    &\frac{d}{dt}F^{n+1} +\mu\left|\nabla\overline{\bu}^{n+1}\right|_{L^2}^2\\
    \le& \bigg(C +C(P')+ C(\e)|\rho^{n+1}|_{L^{\infty}} |\nabla\bu^n|_{L^3}^2
    +C|\nabla \bu^n|_{L^{\infty}}+C|\bu^n|_{L^{\infty}}^2 \\
    &\ +C(\e)|\nabla\rho^n|_{L^2}^2 +C(\e)|\nabla\rho^n|_{L^3}^2
     +C(\e)|\rho^n|_{L^{\infty}}^2 +C|f^n|_{H^1_{\omega}}^2 +C|\nabla\bu^n_t|_{L^2}^2\bigg) F^{n+1} +5\e \left|\nabla\overline{\bu}^n\right|_{L^2}^2.
  \end{aligned}
\end{equation}
Solving the above Gronwall inequality in $[0, T_0]$ $(0<T_0\le T_*)$, we obtain
\begin{equation}\label{eq-app-dif-est}
  \sup_{0\le t\le T_0}F^{n+1}(t) +\mu\int_0^{T_0}\left|\nabla\overline{\bu}^{n+1}(t)\right|_{L^2}^2 dt
  \le A(\e, T_0) \int_0^{T_0}\left|\nabla\overline{\bu}^n(t)\right|_{L^2}^2 dt.
\end{equation}
where $A(\e, T_0)$ is given by
\[
 A(\e, T_0):=5\e\exp\left(\int_0^{T_0} \Big( C(\e, \beta, P', C_0, C_1)+C|\nabla\bu^n(t)|_{L^{\infty}}+C|\nabla\bu^n_t(t)|_{L^2}^2\Big) dt \right).
\]
We first choose $\e$ sufficiently small such that
\[
 5\e\exp\left(\int_0^{T_0}C|\nabla\bu^n_t(t)|_{L^2}^2 dt \right)\le \frac{\mu}{4},
\]
and then take $T_0$ suitably small, so that
\[
 \exp\left(\int_0^{T_0} \Big( C(\e, \beta, P', C_0, C_1)+C|\nabla\bu^n(t)|_{L^{\infty}}\Big) dt\right) \le 2.
\]
Thus, we have $A(\e, T_0)\le \frac{\mu}{2}$ and
\begin{equation}\label{eq-app-dif-est-tzro}
  \sup_{0\le t\le T_0}F^{n+1}(t) +\mu\int_0^{T_0}\left|\nabla\overline{\bu}^{n+1}(t)\right|_{L^2}^2 dt
  \le \frac{\mu}{2} \int_0^{T_0}\left|\nabla\overline{\bu}^n(t)\right|_{L^2}^2 dt.
\end{equation}
Summing \eqref{eq-app-dif-est-tzro} for all $n\in \bbn$ gives
\begin{equation}\label{eq-app-dif-est-sum}
  \sup_{0\le t\le T_0}\sum_{n=1}^{\infty}F^n(t) +\frac{\mu}{2}\sum_{n=1}^{\infty} \int_0^{T_0}\left|\nabla\overline{\bu}^n(t)\right|_{L^2}^2 dt
  \le \frac{\mu}{2} \int_0^{T_0}\left|\nabla\overline{\bu}^0(t)\right|_{L^2}^2 dt.
\end{equation}
We deduce from \eqref{eq-app-dif-est-sum} that there exists $(f, \rho, \bu)$ such that
\begin{equation}\label{eq-app-conver-lowoder}
  \begin{aligned}
    &f^n \to f, \quad\text{in $C([0, T_0]; L^1)$, as $n \to \infty$};\\
    &\rho^n \to \rho, \quad \text{in $C([0, T_0]; L^2)$, as $n \to \infty$};\\
    &\bu^n \to \bu, \quad \text{in $L^2(0, T_0; D^1)$, as $n \to \infty$}.
  \end{aligned}
\end{equation}
From \eqref{eq-app-conver-lowoder}, it is easy to show that $(f, \rho, \bu)$ veries \eqref{eq-cs-ns} in the sense of distributions.
\vskip 3mm
\noindent \textbf{Continuity in Time}
\vskip 3mm
By induction, we know \eqref{eq-appro-rhoandrhot-norm} and \eqref{eq-indu-next} hold for all $n\in \bbn$. Using uniqueness of the weak limit, we deduce by \eqref{eq-app-conver-lowoder} that
\begin{equation}\label{eq-app-conver-wek}
  \begin{aligned}
    f^n &\rightharpoonup f, \quad\text{weakly-$\star$ in $L^{\infty}(0, T_0; H^1_{\omega})$, as $n \to \infty$};\\
    \rho^n &\rightharpoonup \rho, \quad\text{weakly-$\star$ in $L^{\infty}(0, T_0; H^1\cap W^{1,q})$, as $n \to \infty$};\\
    \rho^n_t &\rightharpoonup \rho_t, \quad \text{weakly-$\star$ in $L^{\infty}(0, T_0; L^2\cap L^q)$, as $n \to \infty$};\\
    \bu^n &\rightharpoonup \bu, \quad \text{weakly-$\star$ in $L^{\infty}(0, T_0; D^1\cap D^2)$, as $n \to \infty$};\\
    \bu^n_t &\rightharpoonup \bu_t, \quad \text{weakly in $L^2(0, T_0; D^1)$, as $n \to \infty$};\\
    \bu^n &\rightharpoonup \bu, \quad \text{weakly in $L^2(0, T_0; D^{2,q})$, as $n \to \infty$}.
  \end{aligned}
\end{equation}
It follows from \eqref{eq-app-conver-wek} that
\begin{equation}\label{eq-app-conti-wek}
  \begin{aligned}
    &\rho \in C([0, T_0]; L^2\cap L^q)\cap C([0, T_0]; H^1\cap W^{1,q}-W), \\
    &\bu \in C([0, T_0]; D^1)\cap C([0, T_0]; D^2-W),\\
    &\bu_t \in L^2(0, T_0; D^1),
    \quad \bu \in L^2(0, T_0; D^{2,q}).
  \end{aligned}
\end{equation}
Using the regularity of $\bu$, we can also demonstrate that
\[
 f \in C([0, T_0]; H^1_{\omega})
\]
by the same proof as in Proposition \ref{prop-kine-cs-wp}. Similarly as the proof of \eqref{eq-appro-masswoner-gron}, we can show that
\begin{equation}\label{eq--masswoner-gron}
    \frac{d}{dt}|\rho(t)|_{W^{1,r}}
    \le C(r)|\nabla\bu(t)|_{W^{1,q}}|\rho(t)|_{W^{1,r}}.
\end{equation}
For any $t_1, t_2 \in [0, T_0]$ $(t_1\le t_2)$, it follows from \eqref{eq--masswoner-gron} that
\[
\begin{aligned}
 \Big||\rho(t_2)|_{W^{1,r}}-|\rho(t_1)|_{W^{1,r}} \Big|\le&\int_{t_1}^{t_2}C(r)|\nabla\bu(t)|_{W^{1,q}}|\rho(t)|_{W^{1,r}}dt\\
 \le& C(r, T_0, C_0, C_1)|t_2-t_1|^{\frac12}, \quad 2\le r \le 6.
\end{aligned}
\]
This, together with the fact that $\rho \in C([0, T_0]; H^1\cap W^{1,q}-W)$, implies that
\[
 \rho \in C([0, T_0]; H^1\cap W^{1,q}).
\]
By the regularity of $f$, $\rho$ and $\bu$, we can easily prove
\begin{equation}\label{eq-conti-ltwo}
  \begin{aligned}
    &\int_{\bbr^3}f(\bv-\bu)d\bv \in C([0, T_0]; L^2), \\
    &\nabla P \in C([0, T_0]; L^2), \\
    &\rho \bu_t \in L^2(0, T_0; H^1).
  \end{aligned}
\end{equation}
From $\eqref{eq-cs-ns}_3$, we infer that $(\rho \bu_t)_t \in L^2(0, T_0; H^{-1})$. This together with $\eqref{eq-conti-ltwo}_3$ gives
\begin{equation}\label{eq-conti-rutltwo}
 \rho \bu_t  \in C([0, T_0]; L^2).
\end{equation}
For $0\le t_1\le t_2\le T_0$, it follows from the elliptic estimate that
\begin{equation}\label{eq-conti-usec}
  \begin{aligned}
    |\bu(t_2)-\bu(t_1)|_{D^2}\le& C|\rho\bu_t(t_2)-\rho\bu_t(t_1)|_{L^2} +C|\rho\bu\cdot\nabla\bu(t_2)-\rho\bu\cdot\nabla\bu(t_1)|_{L^2}\\
     &+C|\nabla P(t_2)-\nabla P(t_1)|_{L^2} \\ &+C\left|\int_{\bbr^3}f(t_2)\Big(\bv-\bu(t_2)\Big)d\bv-\int_{\bbr^3}f(t_2) \Big(\bv-\bu(t_1)\Big)d\bv\right|_{L^2},
  \end{aligned}
\end{equation}
where
\begin{equation}\label{eq-conti-rugu}
  \begin{aligned}
    &|\rho\bu\cdot\nabla\bu(t_2)-\rho\bu\cdot\nabla\bu(t_1)|_{L^2}\\
    \le& |\rho(t_2)\bu(t_2)\cdot\Big(\nabla\bu(t_2)-\nabla\bu(t_1)\Big)|_{L^2} +|\rho(t_2)\Big(\bu(t_2)-\bu(t_1)\Big)\cdot\nabla\bu(t_1)|_{L^2}\\ &+|\Big(\rho(t_2)-\rho(t_1)\Big)\bu(t_1)\cdot\nabla\bu(t_1)|_{L^2}\\
    \le&|\rho(t_2)|_{L^{\infty}}|\nabla\bu(t_2)|_{L^2}|\nabla\bu(t_2)-\nabla\bu(t_1|^{\frac12}_{L^2} |\nabla\bu(t_2)-\nabla\bu(t_1|^{\frac12}_{H^1}\\
    &+|\rho(t_2)|_{L^{\infty}}|\nabla\bu(t_1)|_{L^3}|\nabla\bu(t_2)-\nabla\bu(t_1|_{L^2} +|\nabla\bu(t_1)|_{L^2}|\nabla\bu(t_1)|_{L^3}|\rho(t_2)-\rho(t_1|_{L^{\infty}}\\
    \le& C(C_0, C_1)|\rho(t_2)-\rho(t_1|_{W^{1,q}} +C(C_0, C_1)|\bu(t_2)-\bu(t_1|_{D^1} +\frac12 |\bu(t_2)-\bu(t_1|_{D^2}.
  \end{aligned}
\end{equation}
Substituting \eqref{eq-conti-rugu} into \eqref{eq-conti-usec}, we infer by $\eqref{eq-app-conti-wek}_2$, \eqref{eq-conti-ltwo} and \eqref{eq-conti-rutltwo} that
\begin{equation}\label{eq-conti-gtultwo}
  \bu \in C([0, T_0]; D^2).
\end{equation}
The uniqueness of strong solutions can be proved in the same way as in the proof of \eqref{eq-app-dif-gron}.
This completes the proof.$\hfill \square$
%
%
\section{Blowup Criterion for the Coupled System}\label{sec-blowup}
\setcounter{equation}{0}
In this section, we derive a blowup criterion  for the coupled system, which gives an insight into studying the existence of global-in-time strong solutions to the system \eqref{eq-cs-ns}. Our result shows that the $L^{\infty}$-norm of $\rho(t, \bx)$ in $[0, T^*)\times \bbr^3$ and $\int_0^{T^*}\Big(|\bu(t)|_{L^{\infty}}+|\nabla\bu(t)|_{L^{\infty}}^2\Big)dt$ control blowup of the strong solutions at $T^*$. The philosophy of the proof for Theorem \ref{thm-blowup} is that if the blowup mechanism is avoided, then we show that the strong solution can be extended beyond $T^*$, by using Theorem \ref{thm-loc-exist}. The following lemma is the elementary energy estimate for the strong solutions to \eqref{eq-cs-ns}.
Define the energy of the system as
\[
 E(t):=\int_{\bbr^3}\Bigg(\frac12 \rho(t,\bx)\bu^2(t,\bx)+\frac{P(\rho(t,\bx))}{\gamma-1}\Bigg)d\bx +\frac12 \int_{\bbr^6} f(t,\bx,\bv)\bv^2 d\bx d\bv,
\]
and the initial energy $E_0:=E(0)$.
\begin{lemma}\label{lm-ele-enerest}
If $f(t, \bx, \bv)\in C([0, T^*); H^1_{\omega})$, $\rho(t, \bx)\in C([0, T^*); H^1\cap W^{1, q})$, $\bu(t, \bx)\in C([0, T^*); D^1\cap D^2)\cap L^2(0, T^*; D^{2, q})$ is a strong solution to \eqref{eq-cs-ns}-\eqref{eq-sys-inidata}, then it holds for $t\in [0, T^*)$ that
\begin{equation}\label{eq-cs-ns-enconser}
 \begin{gathered}
   E(t)+\int_0^t \Big(\mu|\nabla \bu(\tau)|_{L^2}^2+(\mu+\lambda)|\nabla \cdot \bu(\tau)|_{L^2}^2\Big)d \tau +\int_0^t \int_{\bbr^6} f(\tau,\bx,\bv)(\bu-\bv)^2 d\bx d\bv d\tau\\
   +\frac12\int_0^t \int_{\bbr^6}\int_{\bbr^6}\varphi(|\bx-\by|)f(\tau,\by,\bv^*)f(\tau,\bx,\bv)(\bv^*-\bv)^2d\by d\bv^* d\bx d\bv d\tau=E_0.
 \end{gathered}
\end{equation}
\end{lemma}
\begin{proof}
It follows from $f_0\in H^1_{\omega}$ that
\begin{equation}\label{eq-cs-ener-ini}
  \begin{aligned}
    \int_{\bbr^6}f_0\bv^2d\bx d\bv =& \int_{\bbr^6}f_0\bv^2(1+\bx^2+\bv^2)^{\frac{\alpha}{2}} (1+\bx^2+\bv^2)^{-\frac{\alpha}{2}}d\bx d\bv\\
    \le& |(1+\bx^2+\bv^2)^{-\frac{\alpha}{2}}|_{L^2}\left(\int_{\bbr^6}|f_0|^2\bv^4(1+\bx^2+\bv^2)^{\alpha} d\bx d\bv\right)^{\frac12}\\
    \le& C|f_0|_{L^2_{\omega}}.
  \end{aligned}
\end{equation}
Multiplying $\eqref{eq-cs-ns}_1$ by $\frac12 \bv^2$, and integrating the resulting equation over $\bbr^3\times\bbr^3$ lead to
\begin{equation}\label{eq-cs-ener-dt}
  \begin{gathered}
    \frac{d}{dt}\int_{\bbr^6}\frac12 f\bv^2d\bx d\bv +\frac12 \int_{\bbr^6}\int_{\bbr^6}\varphi(|\bx-\by|)f(t,\by,\bv^*)f(t,\bx,\bv)(\bv^*-\bv)^2d\by d\bv^* d\bx d\bv \\
    =\int_{\bbr^6}f\bv\cdot(\bu-\bv)d\bx d\bv.
  \end{gathered}
\end{equation}
Multiplying $\eqref{eq-cs-ns}_3$ by $\bu$, and integrating the resulting equation over $\bbr^3$ give
\begin{equation}\label{eq-ns-ener-dt}
  \begin{gathered}
    \frac{d}{dt}\int_{\bbr^3}\Bigg(\frac12 \rho\bu^2+\frac{P(\rho)}{\gamma-1}\Bigg)d\bx +\mu|\nabla \bu(t)|_{L^2} +(\mu+\lambda)|\nabla \cdot \bu(t)|_{L^2}\\
    =\int_{\bbr^6}f\bu\cdot(\bv-\bu)d\bx d\bv,
  \end{gathered}
\end{equation}
where we have used the following equality
\[
 \int_{\bbr^3}\nabla P\cdot\bu d\bx =-\int_{\bbr^3}P \nabla\cdot\bu d\bx=\frac{d}{dt}\int_{\bbr^3} \frac{P}{\gamma-1} d\bx.
\]
Adding \eqref{eq-cs-ener-dt} to \eqref{eq-ns-ener-dt}, and integrating the resulting equation over $[0,t]$, $0\le t< T^*$, result in our conclusion \eqref{eq-cs-ns-enconser}. This completes the proof.
\end{proof}
Next we present the proof of \ref{thm-blowup} by contradiction. Suppose
\begin{equation}\label{eq-blowup-antassump}
 |\rho(t,\bx)|_{L^{\infty}(0,T^*;L^{\infty})}+\int_0^{T^*}(|\bu(t)|_{L^{\infty}}+|\nabla \bu(t)|_{L^{\infty}}^2) dt \le C(T^*)<\infty.
\end{equation}
It suffices to show that $(f(T^*,\bx, \bv), \rho(T^*,\bx, \bv), \bu(T^*,\bx, \bv))$ satisfies the initial conditions in Theorem \ref{thm-loc-exist}.
\vskip 3mm
\noindent\textit{Proof of Theorem \ref{thm-blowup}}.Using \eqref{eq-cs-ns-enconser}, we know that
\begin{equation}\label{eq-besti}
 \begin{aligned}
   |\mathbf{b}(t,\bx)|=&\int_{\bbr^{2d}} \varphi(|\bx-\by|)f(t, \by,\bv^*) \bv^* d \by d \bv^*\\
   \le & \left(\int_{\bbr^{2d}} f(t, \by,\bv^*) d \by d \bv^* \right)^{\frac12}\left(\int_{\bbr^{2d}} f(t, \by,\bv^*) |\bv^*|^2 d \by d \bv^* \right)^{\frac12}\\
   \le &C(C_0, E_0), \quad 0\le t < T^*.
 \end{aligned}
\end{equation}
By \eqref{eq-exnV}, it follows from \eqref{eq-blowup-antassump} and \eqref{eq-besti} that
\begin{equation}\label{eq-rtsta}
 R(T^*)\le R_0+ \int_0^{T^*}\Big(C(C_0, E_0)+|\bu(t)|_{\infty}\Big)dt \le C(C_0,E_0, T^*).
\end{equation}
We use Proposition \ref{prop-kine-cs-wp} (2), \eqref{eq-blowup-antassump} and \eqref{eq-rtsta} to deduce that
\begin{equation}\label{eq-ftsta}
 \begin{aligned}
 \sup_{0\le t<T^*}|f(t)|_{H_{\omega}^1}\le& |f_0|_{H_{\omega}^1}\exp \left( C(C_0,E_0, T^*)\int_0^{T^*}(1+|\bu(t)|_{L^{\infty}}+|\nabla\bu(t)|_{L^{\infty}})d t \right)\\ \le& C(C_0,E_0, T^*).
 \end{aligned}
\end{equation}
Denote by $\Dot{\bu}:=\bu_t +\bu\cdot\nabla\bu$ the convective derivative of $\bu$. Multiplying $\eqref{eq-cs-ns}_3$ by $\Dot{\bu}$, and integrating the resulting equation over $\bbr^3$ lead to
\begin{equation}\label{eq-gu-gron}
 \begin{aligned}
   &\frac{d}{dt}\Big(\mu|\nabla \bu|_{L^2}^2+(\mu+\lambda)|\nabla\cdot\bu|_{L^2}^2\Big)+|\rho^{\frac12}\Dot{\bu}|_{L^2}^2\\
   =&\int_{\bbr^3}\Big(\mu\bu\cdot\nabla\bu\cdot\Delta\bu +(\mu+\lambda)\bu\cdot\nabla\bu\cdot\nabla\nabla\cdot\bu\Big)d\bx\\
   & -\int_{\bbr^3}\nabla P\cdot \dot{\bu}d\bx
   +\int_{\bbr^6}f(\bv-\bu)\cdot\dot\bu d\bv d\bx\\
   =:&\sum_{i=1}^3Q_i.
 \end{aligned}
\end{equation}
We estimate each $Q_i$ $(1\le i\le 3)$ as follows.
\begin{equation*}
 \begin{aligned}
   Q_1=&\int_{\bbr^3}\Big(\mu\bu\cdot\nabla\bu\cdot\Delta\bu +(\mu+\lambda)\bu\cdot\nabla\bu\cdot\nabla\nabla\cdot\bu\Big)d\bx\\
   =&\mu\int_{\bbr^3}u_i\partial_iu_j\partial_{kk}u_jd\bx +(\mu+\lambda)\int_{\bbr^3}u_i\partial_iu_j \partial_j\partial_ku_k d\bx\\
   =&-\mu\int_{\bbr^3}\partial_ku_i \partial_iu_j \partial_ku_j d\bx +\frac{\mu}{2}\int_{\bbr^3}\nabla \cdot\bu |\nabla\bu|^2 d\bx\\
   &-(\mu+\lambda)\int_{\bbr^3}\partial_ju_i \partial_iu_j \partial_ku_k d\bx +\frac{(\mu+\lambda)}{2}\int_{\bbr^3}|\nabla \cdot\bu|^3 d\bx\\
   \le& C|\nabla\bu|_{L^{\infty}}|\nabla\bu|_{L^2}^2;
\end{aligned}
\end{equation*}
\begin{equation*}
 \begin{aligned}
   Q_2=&-\int_{\bbr^3}\nabla P\cdot \dot\bu d\bx\\
   =&-\int_{\bbr^3}(\bu_t +\bu\cdot \nabla\bu)\cdot \nabla P  d\bx\\
   =&\frac{d}{dt}\int_{\bbr^3}P\nabla\cdot\bu d\bx +\int_{\bbr^3}\Big(\bu\cdot\nabla P\nabla\cdot\bu +\gamma P|\nabla\cdot\bu|^2 -\bu\cdot \nabla\bu\cdot \nabla P\Big)  d\bx\\
   =&\frac{d}{dt}\int_{\bbr^3}P\nabla\cdot\bu d\bx +(\gamma-1)\int_{\bbr^3}P|\nabla\cdot\bu|^2 d\bx +\int_{\bbr^3}P \partial_iu_j \partial_ju_i d\bx\\
   \le&\frac{d}{dt}\int_{\bbr^3}P\nabla\cdot\bu d\bx + C(P)|\nabla\bu|_{L^2}^2;
\end{aligned}
\end{equation*}
\begin{equation*}
 \begin{aligned}
   Q_3=&\int_{\bbr^6}f(\bu_t +\bu\cdot \nabla\bu)\cdot (\bv-\bu)d\bv d\bx\\
   =&\int_{\bbr^3}\int_{\bbr^3}f\bv d\bv \cdot \bu_t d\bx -\int_{\bbr^3}\int_{\bbr^3}f d\bv \bu\cdot\bu_t d\bx +\int_{\bbr^3}\bu\cdot\nabla\bu\cdot\int_{\bbr^3}f(\bv-\bu)d\bv d\bx\\
   =&\left(\int_{\bbr^6}f\bv\cdot\bu d\bv d\bx\right)_t -\int_{\bbr^3}\int_{\bbr^3}f\bv\otimes\bv d\bv :\bu d\bx\\
   & -\int_{\bbr^3}\int_{\bbr^3}\Big(fL[f]+f(\bu-\bv)\Big)d\bv \cdot \bu d\bx\\
   & -\frac12 \left(\int_{\bbr^6} f\bu^2 d\bv d\bx\right)_t
    +\int_{\bbr^3}\int_{\bbr^3}f\bv d\bv \cdot\nabla\bu \cdot\bu d\bx \\
   &+\int_{\bbr^3} \bu \cdot\nabla\bu \cdot \int_{\bbr^3}f\bv d\bv d\bx -\int_{\bbr^3}\int_{\bbr^3}f d\bv \bu\cdot\nabla\bu \cdot\bu d\bx\\
   \le& \left(\int_{\bbr^6}f\bv\cdot\bu d\bv d\bx -\frac12 \int_{\bbr^6} f\bu^2 d\bv d\bx\right)_t +\left|\int_{\bbr^3}f\bv^2 d\bv\right|_{L^2}|\nabla \bu|_{L^2} \\ &+C|f(1+\bv^2)^{\frac12}|_{L^1} \left|\int_{\bbr^3}f(1+\bv^2)^{\frac12} d\bv\right|_{L^{\frac65}} |\nabla \bu|_{L^2} +\left|\int_{\bbr^3}f\bv d\bv\right|_{L^{\frac65}} |\nabla \bu|_{L^2}\\
   &+2\left|\int_{\bbr^3}f\bv d\bv\right|_{L^3}|\nabla \bu|_{L^2}^2 +\left|\int_{\bbr^3}f d\bv\right|_{L^6}|\nabla \bu|_{L^2}^3\\
   \le& \left(\int_{\bbr^6}f\bv\cdot\bu d\bv d\bx -\frac12 \int_{\bbr^6} f\bu^2 d\bv d\bx\right)_t \\
   &+C\Big(1 +|f|_{H^1_{\omega}}^2 +|\nabla_{\bx}f|_{L^2_{\omega}}^2|\nabla \bu|_{L^2}^2 \Big)|\nabla \bu|_{L^2}^2 +|f|_{L^2_{\omega}}^2.
\end{aligned}
\end{equation*}
Substituting these estimates into \eqref{eq-gu-gron}, we obtain by \eqref{eq-blowup-antassump} and \eqref{eq-ftsta} that
\begin{equation}\label{eq-gu-ineq-gron}
 \begin{aligned}
   &\frac{d}{dt}\Bigg(\mu|\nabla \bu|_{L^2}^2+(\mu+\lambda)|\nabla\cdot\bu|_{L^2}^2 -\int_{\bbr^3}P\nabla\cdot\bu d\bx\\&  -\int_{\bbr^6}f\bv\cdot\bu d\bv d\bx +\frac12 \int_{\bbr^6} f\bu^2 d\bv d\bx \Bigg) +|\rho^{\frac12}\Dot{\bu}|_{L^2}^2\\
   \le& C(C_0,E_0, T^*)\Big(1+|\nabla\bu|_{L^{\infty}}+|\nabla \bu|_{L^2}^2\Big)|\nabla \bu|_{L^2}^2 +C(C_0,E_0, T^*),
 \end{aligned}
\end{equation}
holds for $t \in [0, T^*)$.

Applying $\bigg(\frac{\partial}{\partial t}\star +\nabla\cdot(\bu\otimes\star)\bigg)\cdot\dot\bu$ to $\eqref{eq-cs-ns}_3$ and integrating the resulting equation over $\bbr^3$, we obtain
\begin{equation}\label{eq-rudot-dt}
  \begin{aligned}
    &\frac12 \frac{d}{dt}\int_{\bbr^3}\rho\dot\bu^2 d\bx +\mu|\nabla\dot\bu|_{L^2}^2 +(\mu+\lambda)|\nabla\cdot\dot\bu|_{L^2}^2\\
    =&-\int_{\bbr^3}\Big(\nabla P_t +\nabla\cdot(\bu\otimes\nabla P)\Big)\cdot\dot\bu d\bx\\
    &+\mu \int_{\bbr^3}\Big(\nabla\cdot(\bu\otimes\Delta \bu) -\Delta(\bu\cdot\nabla\bu) \Big)\cdot\dot\bu d\bx\\
    &+(\mu+\lambda)\int_{\bbr^3}\Big(\nabla\cdot(\bu\otimes\nabla\nabla\cdot\bu) -\nabla\nabla\cdot(\bu\cdot\nabla\bu)\Big)\cdot\dot\bu d\bx\\
    &+\int_{\bbr^3}\int_{\bbr^3}f_t(\bv-\bu)d\bv\cdot\dot\bu d\bx -\int_{\bbr^3}\int_{\bbr^3}f d\bv \bu_t\cdot \dot\bu d\bx\\
    &+\int_{\bbr^3}\nabla\cdot\bigg(\bu\otimes\int_{\bbr^3}\int_{\bbr^3}f(\bv-\bu)d\bv\bigg)\cdot\dot\bu d\bx\\
    =:&\sum_{i=1}^6S_i.
  \end{aligned}
\end{equation}
We estimate each $S_i$ $(1\le i \le 6)$ as follows.
\begin{equation*}
 \begin{aligned}
   S_1=&-\int_{\bbr^3}\Big(\nabla P_t +\nabla\cdot(\bu\otimes\nabla P)\Big)\cdot\dot\bu d\bx\\
   =&\int_{\bbr^3}\Big(P_t \nabla\cdot\dot\bu -\nabla\cdot\bu\nabla P\cdot\dot\bu -\bu\cdot\nabla\nabla P\cdot\dot\bu\Big) d\bx\\
   =&-\int_{\bbr^3} P \partial_ju_i \partial_i\dot u_j d\bx +(\gamma-1)\int_{\bbr^3} P\nabla\cdot\bu \nabla\cdot\dot\bu d\bx\\
   \le& C(P)|\nabla\cdot\bu|_{L^2}|\nabla\cdot\dot\bu|_{L^2}\\
   \le& \frac{\mu}{12}|\nabla\dot\bu|_{L^2}^2 +C(P)|\nabla\bu|_{L^2}^2;
\end{aligned}
\end{equation*}
\begin{equation*}
 \begin{aligned}
   S_2=&\mu \int_{\bbr^3}\Big(\nabla\cdot(\bu\otimes\Delta \bu) -\Delta(\bu\cdot\nabla\bu) \Big)\cdot\dot\bu d\bx\\
   =&\mu \int_{\bbr^3}\Big(\partial_ju_j \partial_{ii}u_k \dot u_k + u_j\partial_j\partial_{ii}u_k \dot u_k -\partial_{ii}(u_j\partial_ju_k)\dot u_k \Big) d\bx\\
   =&\mu \int_{\bbr^3}\Big(\partial_iu_j \partial_ju_k \partial_i\dot u_k +\partial_iu_j \partial_iu_k \partial_j\dot u_k- \partial_ju_j \partial_iu_k \partial_i\dot u_k \Big) d\bx\\
   \le& \frac{\mu}{12}|\nabla\dot\bu|_{L^2}^2 +C|\nabla \bu|_{\infty}^2|\nabla\bu|_{L^2}^2;
\end{aligned}
\end{equation*}
\begin{equation*}
 \begin{aligned}
   S_3=&(\mu+\lambda) \int_{\bbr^3}\Big(\nabla\cdot(\bu\otimes\nabla\nabla\cdot\bu) -\nabla\nabla\cdot(\bu\cdot\nabla\bu)\Big)\cdot\dot\bu d\bx\\
   =&(\mu+\lambda) \int_{\bbr^3}\Big(\partial_iu_i \partial_{kj}u_j \dot u_k + u_i\partial_i\partial_{kj}u_j \dot u_k -\partial_{j}(u_i\partial_iu_j)\partial_k\dot u_k \Big) d\bx\\
   =&(\mu+\lambda) \int_{\bbr^3}\Big(\partial_ju_i \partial_iu_j \partial_k\dot u_k +\partial_ku_i \partial_ju_j \partial_i\dot u_k- \partial_iu_i \partial_ju_j \partial_k\dot u_k \Big) d\bx\\
   \le& \frac{\mu}{12}|\nabla\dot\bu|_{L^2}^2 +C|\nabla \bu|_{\infty}^2|\nabla\bu|_{L^2}^2;
\end{aligned}
\end{equation*}
\begin{equation*}
 \begin{aligned}
   S_4=&\int_{\bbr^3}\int_{\bbr^3}f_t(\bv-\bu)d\bv\cdot\dot\bu d\bx \\
   =&-\int_{\bbr^3}\int_{\bbr^3}f \bv d\bv \cdot\nabla\bu\cdot\dot\bu d\bx +\int_{\bbr^3}\int_{\bbr^3}f \bv \otimes (\bv-\bu) d\bv : \nabla\dot\bu d\bx\\
   &+\int_{\bbr^3}\int_{\bbr^3}\Big(fL[f] +f(\bu-\bv) d\bv \dot\bu d\bx\\
   \le& 2\left|\int_{\bbr^3}f\bv d\bv\right|_{L^3}|\nabla \bu|_{L^2}|\nabla\dot \bu|_{L^2} +\left|\int_{\bbr^3}f\bv^2 d\bv\right|_{L^2}|\nabla\dot \bu|_{L^2}\\
   &+C|f(1+\bv^2)^{\frac12}|_{L^1} \left|\int_{\bbr^3}f(1+\bv^2)^{\frac12} d\bv\right|_{L^{\frac65}} |\nabla \dot\bu|_{L^2} +\left|\int_{\bbr^3}f d\bv\right|_{L^{\frac32}} |\nabla \bu|_{L^2}|\nabla\dot \bu|_{L^2}\\
   \le& \frac{\mu}{12}|\nabla\dot\bu|_{L^2}^2 +C|f|_{H^1_{\omega}}^2|\nabla\bu|_{L^2}^2 +C|f|_{L^2_{\omega}}^2 +C|f|_{L^2_{\omega}}^4;\\[4mm]
   S_5=&-\int_{\bbr^3}\int_{\bbr^3}f d\bv \bu_t\cdot \dot\bu d\bx\\
   =&-\int_{\bbr^3}\int_{\bbr^3}f d\bv (\dot\bu -\bu\cdot\nabla\bu)\cdot\dot\bu d\bx \\
   \le&\int_{\bbr^3}\int_{\bbr^3}f d\bv \bu\cdot\nabla\bu\cdot\dot\bu d\bx \\
   \le& \left|\int_{\bbr^3}f d\bv\right|_{L^6}|\nabla \bu|_{L^2}^2|\nabla\dot\bu|_{L^2}\\
   \le& \frac{\mu}{12}|\nabla\dot\bu|_{L^2}^2 +C|f|_{H^1_{\omega}}^2|\nabla\bu|_{L^2}^4;
\end{aligned}
\end{equation*}
\begin{equation*}
 \begin{aligned}
   S_6=&\int_{\bbr^3}\nabla\cdot\bigg(\bu\otimes\int_{\bbr^3}\int_{\bbr^3}f(\bv-\bu)d\bv\bigg)\cdot\dot\bu d\bx\\
   =&\int_{\bbr^3}\nabla\cdot\bu \int_{\bbr^3}\int_{\bbr^3}f(\bv-\bu)d\bv\cdot\dot\bu d\bx +\int_{\bbr^3}\bu\cdot\nabla\int_{\bbr^3}f(\bv-\bu)d\bv\cdot\dot\bu d\bx\\
   =&-\int_{\bbr^3}\bu\otimes\int_{\bbr^3}f(\bv-\bu)d\bv : \nabla\dot\bu d\bx\\
   \le& \left|\int_{\bbr^3}f\bv d\bv\right|_{L^3}|\nabla \bu|_{L^2}|\nabla\dot\bu|_{L^2} +\left|\int_{\bbr^3}f d\bv\right|_{L^6}|\nabla \bu|_{L^2}^2|\nabla\dot\bu|_{L^2}\\
   \le& \frac{\mu}{12}|\nabla\dot\bu|_{L^2}^2 +C|f|_{H^1_{\omega}}^2\Big(1+|\nabla\bu|_{L^2}^2\Big)|\nabla\bu|_{L^2}^2.
\end{aligned}
\end{equation*}
Substituting these estimates into \eqref{eq-rudot-dt}, we obtain by \eqref{eq-blowup-antassump} and \eqref{eq-ftsta} that
\begin{equation}\label{eq-sqrudt-ineq-gron}
   \frac{d}{dt}|\rho^{\frac12}\dot\bu|_{L^2}^2 +\mu|\nabla\dot\bu|_{L^2}^2
   \le C(C_0,E_0, T^*)\Big(1+|\nabla\bu|_{L^{\infty}}^2+|\nabla \bu|_{L^2}^2\Big)|\nabla \bu|_{L^2}^2 +C(C_0,E_0, T^*),
\end{equation}
holds for $t \in [0, T^*)$.
Combining \eqref{eq-gu-ineq-gron} with \eqref{eq-sqrudt-ineq-gron}, we deduce that
\begin{equation}\label{eq-gupsqrudt-ineq-gron}
 \begin{aligned}
   &\frac{d}{dt}\Bigg(|\rho^{\frac12}\dot\bu|_{L^2}^2+\mu|\nabla \bu|_{L^2}^2+(\mu+\lambda)|\nabla\cdot\bu|_{L^2}^2 -\int_{\bbr^3}P\nabla\cdot\bu d\bx\\&  -\int_{\bbr^6}f\bv\cdot\bu d\bv d\bx +\frac12 \int_{\bbr^6} f\bu^2 d\bv d\bx \Bigg) +\mu|\nabla\dot\bu|_{L^2}^2\\
   \le& C(C_0,E_0, T^*)\Big(1+|\nabla\bu|_{L^{\infty}}^2+|\nabla \bu|_{L^2}^2\Big)|\nabla \bu|_{L^2}^2 +C(C_0,E_0, T^*),
 \end{aligned}
\end{equation}
holds for $t \in [0, T^*)$.
Multiplying $\eqref{eq-cs-ns}_2$ by $2\rho$ and integrating the resulting equation equation over $\bbr^3$
lead to
\begin{equation}\label{eq-rdt-gron}
  \frac{d}{dt}|\rho|_{L^2}^2=-\int_{\bbr^3}\rho^2\nabla\cdot\bu d\bx \le |\nabla\bu|_{L^{\infty}}|\rho|_{L^2}^2.
\end{equation}
By \eqref{eq-blowup-antassump}, it follows from \eqref{eq-rdt-gron} that
\begin{equation}\label{eq-rhotstar}
  \sup_{0\le t<T^*}|\rho(t)|_{L^2}\le |\rho_0|_{L^2}\exp\left(\int_0^{T^*}|\nabla\bu(t)|_{L^{\infty}}d t\right) \le C(C_0, T^*).
\end{equation}
Using \eqref{eq-blowup-antassump}, \eqref{eq-ftsta} and \eqref{eq-rhotstar}, we have for $0\le t<T^*$
\begin{equation}\label{eq-pdivu}
 \begin{aligned}
  \int_{\bbr^3}P\nabla\cdot\bu d\bx \le& \frac{\mu}{4}|\nabla \bu|_{L^2}^2 +C(P')|\rho|_{L^2}^2\\
   \le& \frac{\mu}{4}|\nabla \bu|_{L^2}^2 +C(C_0, T^*),
  \end{aligned}
\end{equation}
and
\begin{equation}\label{eq-intfvdotu}
 \begin{aligned}
  \int_{\bbr^3}\int_{\bbr^3}f\bv d\bv\cdot\bu d\bx \le& \frac{\mu}{4}|\nabla \bu|_{L^2}^2 +C\left|\int_{\bbr^3}f\bv d\bv\right|_{L^{\frac65}} ^2\\
    \le& \frac{\mu}{4}|\nabla \bu|_{L^2}^2 +C|f|_{L^2_{\omega}}^2\\
   \le& \frac{\mu}{4}|\nabla \bu|_{L^2}^2 +C(C_0, E_0, T^*).
  \end{aligned}
\end{equation}
Employing \eqref{eq-cs-ns-enconser}, \eqref{eq-blowup-antassump}, \eqref{eq-pdivu},  \eqref{eq-intfvdotu} and the compatibility condition, solving the Gronwall inequality \eqref{eq-gupsqrudt-ineq-gron} gives
\begin{equation}\label{eq-sqrudtpgutstar}
  \sup_{0\le t<T^*}\Big(|\rho^{\frac12}\dot\bu(t)|_{L^2}^2+|\nabla\bu(t)|_{L^2}^2\Big)
  +\int_0^{T^*}|\nabla\dot\bu(t)|_{L^2}^2 d t \le C(C_0,E_0, T^*).
\end{equation}
Applying $\nabla$ to $\eqref{eq-cs-ns}_2$, we obtain
\begin{equation}\label{eq-grt}
  (\nabla\rho)_t +\nabla\bu\cdot\nabla\rho+ \bu\cdot\nabla\nabla\rho +\nabla\rho\nabla\cdot\bu +\rho\nabla\nabla\cdot\bu=0.
\end{equation}
Multiplying \eqref{eq-grt} by $6|\nabla\rho|^4\nabla\rho$ gives rise to
\begin{equation}\label{eq-grsixt}
  \begin{gathered}
  \Big(|\nabla\rho|^6\Big)_t +\nabla\cdot\Big(|\nabla\rho|^6\bu\Big)+ 5|\nabla\rho|^6\nabla\cdot\bu + 6|\nabla\rho|^4\nabla\rho \cdot\nabla\bu \cdot\nabla\rho\\
   +6\rho|\nabla\rho|^4\nabla\rho \cdot\nabla\nabla\cdot\bu=0.
  \end{gathered}
\end{equation}
Integrating \eqref{eq-grsixt} over $\bbr^3$, we infer that
\begin{equation}\label{eq-grsixt-gron}
  \frac{d}{dt}|\nabla\rho|_{L^6}\le C|\nabla\bu(t)|_{L^{\infty}}|\nabla\rho|_{L^6} +C|\rho|_{L^{\infty}}|\nabla^2\bu|_{L^6}.
\end{equation}
Using the elliptic estimates, it follows from $\eqref{eq-cs-ns}_3$ that
\begin{equation}\label{eq-gtwousix}
 \begin{aligned}
  |\nabla^2\bu|_{L^6}\le& C\Bigg(|\rho\dot\bu|_{L^6} +|\nabla P|_{L^6}+ \left|\int_{\bbr^3}f(\bv-\bu)d\bv\right|_{L^6}\Bigg)\\
  \le& C |\rho|_{L^{\infty}} |\nabla\dot\bu|_{L^2} +C(P')|\nabla\rho|_{L^6}
     +C|\nabla_{\bx}f|_{L^2_{\omega}} +C |\bu|_{L^{\infty}}|\nabla_{\bx}f|_{L^2_{\omega}}.
  \end{aligned}
\end{equation}
Substituting \eqref{eq-gtwousix} into \eqref{eq-grsixt}, and using \eqref{eq-blowup-antassump},
 \eqref{eq-ftsta},we obtain for $0\le t<T^*$ that
\begin{equation}\label{eq-grsixtest-gron}
  \frac{d}{dt}|\nabla\rho|_{L^6}\le C(T^*)\Big(1+|\nabla\bu(t)|_{L^{\infty}}\Big)|\nabla\rho|_{L^6} +C(C_0, E_0, T^*)\Big(1+|\bu|_{L^{\infty}}+|\nabla\dot\bu|_{L^2}\Big).
\end{equation}
Employing \eqref{eq-blowup-antassump} and \eqref{eq-sqrudtpgutstar}, solving the above Gronwall inequality gives
\begin{equation}\label{eq-grsixtstar}
  \sup_{0\le t<T^*}|\nabla\rho(t)|_{L^6}\le C(C_0, E_0, T^*).
\end{equation}
By the Gagliardo--Nirenberg inequality, we have for $2\le r \le 6$
\[
  |\rho|_r\le |\rho|_{L^2}^{\theta_1}|\nabla\rho|_{L^6}^{1-\theta_1}\le |\rho|_{L^2}+ |\nabla\rho|_{L^6},
\]
and
\[
  |\nabla\rho|_r\le |\rho|_{L^2}^{\theta_2}|\nabla\rho|_{L^6}^{1-\theta_2}\le |\rho|_{L^2}+ |\nabla\rho|_{L^6},
\]
where
\[
 -\frac{3}{r}=-\frac{3\theta_1}{2}+\frac{1-\theta_1}{2},\quad 1-\frac{3}{r}=-\frac{3\theta_2}{2}+\frac{1-\theta_2}{2}.
\]
From \eqref{eq-rhotstar},  \eqref{eq-grsixtstar} and the above interpolation, we know that
\begin{equation}\label{eq-rsobotstar}
  \sup_{0\le t<T^*}|\rho(t)|_{H^1\cap W^{1,q}}\le C(C_0, E_0, T^*).
\end{equation}
Using the elliptic estimates again, it follows from $\eqref{eq-cs-ns}_3$ that
\begin{equation}\label{eq-gtwoutwo}
 \begin{aligned}
  |\nabla^2\bu|_{L^2}\le& C\Bigg(|\rho\dot\bu|_{L^2} +|\nabla P|_{L^2}+ \left|\int_{\bbr^3}f(\bv-\bu)d\bv\right|_{L^2}\Bigg)\\
  \le& C |\rho|_{L^{\infty}}^{\frac12} |\rho^{\frac12}\dot\bu|_{L^2} +C(P')|\nabla\rho|_{L^2}
     +C|f|_{L^2_{\omega}} +C |f|_{H^1_{\omega}}|\nabla\bu|_{L^2}.
  \end{aligned}
\end{equation}
By virtue of \eqref{eq-blowup-antassump}, \eqref{eq-ftsta}, \eqref{eq-sqrudtpgutstar} and \eqref{eq-rsobotstar}, we infer that
\begin{equation}\label{eq-gtutstar}
  \sup_{0\le t<T^*}|\bu(t)|_{D^2}\le C(C_0, E_0, T^*).
\end{equation}
In terms of the regularity of $f$, $\rho$ and $\bu$, it is easy to show $\rho \dot\bu \in C([0, T^*);L^2)$.
Define
\[
 \rho \dot\bu(T^*):=\lim_{t\to T^*}\rho \dot\bu(t) \quad \text{in $L^2(\bbr^3)$}
\]
and
\[
 g(T^*):=
  \begin{dcases}
    \rho(T^*)^{\frac12}\rho \dot\bu(T^*), \quad & \rho(T^*)\neq 0,\\
    0,  \quad& \rho(T^*)=0.
  \end{dcases}
\]
From \eqref{eq-sqrudtpgutstar}, we know $g(T^*)\in L^2(\bbr^3)$. Therefore, $\Big(f(T^*), \rho(T^*), \bu(T^*)\Big)$ satisfies all the conditions on the initial data. Thus, we can use Theorem \ref{thm-loc-exist} to extend the local strong solution beyond $T^*$, which contradicts our assumption on the life span. This completes the proof. $\hfill \square$

%
%
\section{Conclusion}
In this paper, we study the local existence and blowup criterion for the strong solutions to the kinetic Cucker--Smale model coupled with the isentropic compressible Navier--Stokes equation. Our result shows that the upper bound of $\rho(t,\bx)$ and the integrability of $|\bu(t)|_{L^{\infty}}$ and $|\nabla\bu(t)|_{L^{\infty}}^2$ control the blowup of the strong solutions to \eqref{eq-cs-ns}. We wish to give an insight into analyzing the existence of global-in-time strong solutions by this criterion.

Up to now, most previous relevant literatures are concentrated on the space-periodic domain, by using the positive lower bound of the interaction kernel or the Poincar\'e inequality in the process of analysis, while in this paper, we contribute a study for the whole space situation. The novelty of this paper is that we introduce a weighted Sobolev space and present a detailed analysis for the kinetic Cucker--Smale model, as well as manage to overcome the difficult estimates arising from the coupling term. Based on our investigation on the blowup mechanism, how to devise initial data to obtain  the global-in-time strong solutions to the coupled system, is an interesting problem deserving our further endeavor.

\end{document}